\documentclass[10pt, leqno]{article}

\usepackage{mathrsfs}
\usepackage{amsmath}
\usepackage{amssymb}
\usepackage{amsthm}
\usepackage{comment}

\newtheorem{definition}{Definition}[section]
\newtheorem{theorem}{Theorem}[section]
\newtheorem{proposition}{Proposition}[section]
\newtheorem{lemma}{Lemma}[section]
\newtheorem{corollary}{Corollary}[section]
\newtheorem{remark}{Remark}[section]

\numberwithin{equation}{section}

\begin{document}

\title{On the Cauchy problem for $D_t^2-D_x\big(b(t)a(x)\big)D_x$}

\author{Ferruccio Colombini \footnote{Dipartimento di Matematica, Universit\`a di Pisa:  ferruccio.colombini@unipi.it}\;\; and Tatsuo Nishitani\footnote{Department of Mathematics, Osaka University:  
nishitani@math.sci.osaka-u.ac.jp
}}

\date{}
\maketitle

\def\dif{\partial}
\def\al{\alpha}
\def\be{\beta}
\def\ga{\gamma}
\def\om{\omega}
\def\lam{\lambda}
\def\tika{{\tilde \kappa}}
\def\baka{{\bar \kappa}}
\def\varep{\varepsilon}
\def\tal{{\tilde\alpha}}
\def\tbe{{\tilde\beta}}
\def\tis{{\tilde s}}
\def\bas{{\bar s}}
\def\R{{\mathbb R}}
\def\N{{\mathbb N}}
\def\Ga{\Gamma}
\def\La{\Lambda}
\def\lr#1{\langle{#1}\rangle}
\def\mD{\lr{ D}_{\mu}}
\def\xim{\lr{\xi}_{\mu}}

\begin{abstract}
We consider the Cauchy problem for second order differential operators with two independent variables $P=D_t^2-D_x(b(t)a(x))D_x$. Assuming that $b(t)$ is a nonnegative $C^{n,\alpha}$ function and $a(x)$ is a nonnegative Gevrey function of order $1<s<1+(n+\alpha)/2$ we prove that the Cauchy problem for $P$ is well-posed in the Gevrey class of any order $s'$ with $1<s<s'<1+(n+\alpha)/2$. 
\end{abstract}

 

\section{Introduction}

We are concerned with the Cauchy problem for  second order differential operators with two independent variables 
\begin{equation}
\label{eqia}
P=D_t^2-D_x\big(b(t) a(x)\big)D_x,\quad D_t=\frac{1}{i}\frac{\dif}{\dif t}, \;\; D_x=\frac{1}{i}\frac{\dif}{\dif x}
\end{equation}
with a nonnegative $b(t)\in C^{n,\al}([0,T])$, with $n\geq 0$ integer and $0\leq \al\leq1$  and a nonnegative Gevrey function $a(x)$  such that
\begin{equation}
\label{eq:mondai}
\left\{\begin{array}{ll}
Pu=0\;\;\mbox{in}\;\;(t,x)\in [0,T']\times \R,\\
D_t^ju(0,x)=u_j(x)  \;\;\mbox{for}\;\;j=0,1
\end{array}
\right.
\end{equation}
where $0<T'\leq T$. In the special case that $a(x)$ is positive constant  the well-posedness in the Gevrey class of order $1<s'<1+(n+\al)/2$ is proved in \cite{CJS}. Moreover this result is optimal in the sense that there exists $0\leq b(t)\in C^{n,\al}([0,T])$ such that the Cauchy problem for $D_t^2-b(t)D_x^2$ is not well-posed in the Gevrey class of order $s'>1+(n+\al)/2$ (for more details see \cite[Theorem 3]{CJS}). We denote by ${\mathcal G}^s(\R)$ the set of functions which are uniformly Gevrey $s$ on $\R$, that is the function $a(x)\in C^{\infty}(\R)$ belongs to ${\mathcal G}^s(\R)$  if  there exist $C>0, A>0$ such that
\[
|\dif_x^ka(x)|\leq CA^kk!^s,\quad \forall x\in \R, \quad k=0,1,\ldots.
\]
Denote ${\mathcal G}_0^s(\R)={\mathcal G}^s(\R)\cap C_0^{\infty}(\R)$. In this paper we prove
\begin{theorem}
\label{thm:taifu}Assume $0\leq b(t)\in C^{n,\al}([0,T])$ and $0\leq a(x)\in {\mathcal G}^s(\R)$. Assume $1<s<s'<1+(n+\al)/2$ then the Cauchy problem \eqref{eq:mondai} for $P$ is well-posed in ${\mathcal G}^{s'}(\R)$. More precisely there exists $\theta_0>0$ such that for any $1<s<s'<1+(n+\al)/2$ and $0<\tau \,(\leq T)$ there is $\mu>0$  such that for any $u_j\in {\mathcal G}_0^{s'}(\R)$ with finite $\sum_{j=0}^1\|\mD^{1-j}{\rm Op}(e^{\tau\xim^{1/s'}})u_j\|$ and any $0<\tau'<\tau$ there exists
a unique solution $u\in C^1([0,\tau'/\theta_0);{\mathcal G}_0^{s'}(\R))$ to \eqref{eq:mondai} with finite $\sum_{j=0}^1\|\mD^{(1-j)/s'}{\rm Op}(e^{(\tau'-\theta_0t)\xim^{1/s'}})D_t^ju(t)\|$.
\end{theorem}
For  a more detailed statement, see Proposition \ref{pro:seikaku} below. For case $b(t)\in C^{n,\alpha}$ with $n=0,1$ the result is a special case of  a more general result \cite[Th\'eor\`eme 1.3]{N:BSM}.
\begin{remark}
\label{rem:tyui}\rm From the Plancherel's theorem one has for $\tau>0$
\[
\|\mD^{\ell}{\rm Op}(e^{\tau\xim^{\kappa}})v\|^2=\sum_{n=0}^{\infty}\frac{(2\tau)^n}{n!}\|\mD^{\ell}\mD^{\kappa n/2}v\|^2.
\]
\end{remark}
\begin{corollary}
\label{cor:Cmugen}Assume $b(t)\in C^{\infty}([0,T])$ and $0\leq a(x)\in {\mathcal G}^s(\R)$ with $1<s$ which is constant for large $|x|$. Then the Cauchy problem for $P$ is well-posed in ${\mathcal G}^{s'}(\R)$ for any $s'>s$.
\end{corollary}
In \cite{CS} they constructed a nonnegative $C^{\infty}$  function $b(t)$ such that the Cauchy problem for $D_t^2-b(t)D_x^2$ is not $C^{\infty}$ well-posed.
Our proof of Theorem \ref{thm:taifu} is based on the energy method. To
explain the idea we consider the case $a(x)\equiv 1$. After Fourier transform with respect to $x$ it suffices to consider $\dif_t^2+b(t)\xi^2$. A very naive idea to obtain energy  estimates is to employ a weighted energy
\[
E(u)=e^{-\La (t,\xi)}\Bigl\{|\dif_t{\hat u}(t,\xi)|^2+(b(t)+|\xi|^{-2\delta})\xi^2
|{\hat u}(t,\xi)|^2\Bigr\}
\]
with weight $e^{-\Lambda}$ where $\delta>0$ to be chosen suitably.
In $dE(u)/dt$ we have a term $e^{-\La(t,\xi)}b'(t)\xi^2|{\hat
u}(t,\xi)|^2$ which is the hardest term to manage because
$b'(t)$ may change the sign although $b(t)\geq 0$ is assumed. A way to cancel the
term is to choose $\La(t,\xi)$ so that
\[
\La(t,\xi)=\int_0^t\frac{|b'(s)|}
{b(s)+|\xi|^{-2\delta}}\,ds
\]
since $dE(u)/dt$ supplies the term
$- e^{-\La}|b'(t)|\xi^2|{\hat u}(t,\xi)|^2$ which obviously 
controls $e^{-\La}b'(t)\xi^2|{\hat u}(t,\xi)|^2$. For general
$D_t^2-D_xh(t,x)D_x$ this suggests to choose
\[
\La(t,x,\xi)=\int_0^t\frac{|\dif_th(s,x)|}{h(s,x)+|\xi|^{-2\delta}}\,ds
\]
which, however, does not work  as a symbol of
pseudodifferential operator for the lack of regularity in $x$ (see however \cite{N1}). If $h(t,x)=b(t)a(x)$, thanks to this special form, the weight $\Lambda(t,x,\xi)$ works as a symbol of pseudodifferential operator with the metric $g=(a(x)+|\xi|^{-2\delta})^{-1}dx^2+|\xi|^{-2}d\xi^2$, which is equivalent to the one defining the symbol class $S_{1,\delta}$ on the set $\{a(x)= 0\}$. However, since an expected weight is not ${\rm Op}(\pm\Lambda)$ but ${\rm Op}(e^{\pm \Lambda})$  we must provide a calculus including ${\rm Op}(e^{\pm\Lambda})$.  Moreover in order to obtain energy estimates in the Gevrey class of order $s'$ we need to choose   $\delta =1-1/s'$ which could be  very close to $1$ and the larger $\delta$ gives worse commutators against $\Lambda$ because $\dif_x\Lambda$ looses the factor $(a(x)+|\xi|^{-2\delta})^{-1/2}$ on the symbol level, which is $|\xi|^{\delta}$ if $a(x)=0$. To manage such commutators we take advantage of pseudodifferential calculus with the metric $g$ (see \cite{N1},  \cite{CN2}).

\section{Metrics, symbols and weights}

In this section we introduce a metric defining a class of symbols of pseudodifferential operators and weights which we use throughout the paper. For Weyl-H\"ormander calculus of pseudodifferential operators we refer to \cite{Ho}, \cite{Ler}.
We denote 
\[
\xim=(\mu^{-2}+|\xi|^2)^{1/2}
\]
where $0<\mu\ll 1$ is a small parameter. Assume that $0\leq a(x)\in {\mathcal G}^s(\R)$ and with $0\leq \delta<1$ we set
\[
\phi(x,\xi)=a(x)+\xim^{-2\delta}.
\]
Throughout the paper  all constants are assumed to be independent of $\mu$ unless  otherwise stated.
\begin{lemma}
\label{lem:katto}
There exist $A>0, C>0$ such that
\begin{equation}
\label{eq:hasaku}
|\dif_x^{k}\dif_{\xi}^{l}\phi(x,\xi)|/\phi\leq CA^{k+l}(k+l)!(1+(k+l)^{s-1}\xim^{-\delta})^{k+l}\xim^{-l+\delta k}
\end{equation}
for any $k, l\in \N$.  
\end{lemma}
\begin{proof}
Since $a(x)$ is nonnegative note that $|\dif^j_xa(x)|\leq Ca^{1-j/2}\leq C\phi^{1-j/2}\leq C\phi\xim^{j\delta}\leq C\phi\xim^{j\delta}$ for $j=0,1,2$ and 
\begin{equation}
\label{eq:sanada}
|\dif_{\xi}^j\xim^{-2\delta}|\leq CA^jj!\xim^{-2\delta}\xim^{-j}\leq CA^jj!\phi\xim^{-j}
\end{equation}
for any $j\in\N$. Then for $k+l\leq 2$ the estimate \eqref{eq:hasaku} holds clearly. For any $k\geq 3$ we have
\[
k!^s\xim^{2\delta}\leq \big(2^{2(s-1)}\big)^k(1+k^{s-1}\xim^{-\delta})^kk!\xim^{\delta k}
\]
which proves with $A_1=2^{2(s-1)}$ that for $k\geq 3$
\[
|\dif_x^ka(x)|\leq CA^kk!^s\phi \xim^{2\delta}\leq C(AA_1)^kk!(1+k^{s-1}\xim^{-\delta})^k\phi\xim^{\delta k}.
\]
This together with \eqref{eq:sanada} proves the assertion for any $k,l\in\N$. 
\end{proof}

Assume $0\leq b(t)\in C^{n,\al}([0,T])$ and consider
\begin{equation}
\label{eq:ramu}
\La(t,x,\xi)=\int^t_0\frac{|b'(s)|\phi(x,\xi)}{b(s)\phi(x,\xi)+\xim^{-2\delta}}\,ds.
\end{equation}
Here recall \cite[Lemma 1]{CJS}.
\begin{lemma}
\label{lem:iia}
Assume that $f(t,\eta)\geq 0$ and $f(t,\eta)\in C^{n,\al}([0,T])$ for any $\eta\in \Omega$. Then one has 
\[
\int_0^t\frac{|\dif_tf(s,\eta)|}{f(s,\eta)+r}ds\leq C_NB^{1/(n+\al)}r^{-1/(n+\al)}
\]
for $(t,\eta)\in [0,T]\times \Omega$ where $B=\sup_{\eta\in\Omega}\|f(t,\eta)\|_{C^{n,\al}([0,T])}$. 
\end{lemma}

Denote $N=n+\al$ and assume $
1<s<s'<1+N/2$. 
Define $\kappa>0$ and $\tika$ by
\[
0<\kappa=\frac{1}{s'}\,,\quad \tika=\frac{2\delta}{N}
\]
so that $s\kappa<s'\kappa=1$ and we put
\[
0<\delta=1-\kappa<1.
\]
\begin{lemma}
\label{lem:jimei}We have ${\tilde \kappa}<\kappa$ and hence $s{\tilde \kappa}<s\kappa=s(1-\delta)<1$.
\end{lemma}
\begin{proof}
Since $s'<1+N/2$ then $\delta=1-1/s'<N/(2+N)$. Thus we have
\[
\kappa=1-\delta>\frac{2}{2+N}=\frac{2}{N}\,\frac{N}{2+N}>\frac{2\delta}{N}=\tika
\]
which proves the assertion. 
\end{proof}
\begin{corollary}
\label{cor:kurinasi}
There is $B$ such that
\begin{equation}
\label{eq:keiho}
\La(t,x,\xi)\leq B\xim^{\tika}.
\end{equation}
\end{corollary}
\begin{proof} Thanks to Lemma \ref{lem:iia} the proof is clear.
\end{proof}
\begin{lemma}
\label{lem:ano:1}There are $A>0, C>0$ such that
\[
\big|\dif_x^k\dif_{\xi}^l\La(t,x,\xi)|\leq CA^{k+l}(k+l)!(1+(k+l)^{s-1}\xim^{-\delta})^{k+l}\La \xim^{-l+\delta k}
\]
for any $k,l\in \N$.
\end{lemma}
\begin{proof} Note that with $\Phi=b(s)\phi(x,\xi)+\xim^{-2\delta}$ one has
\begin{equation}
\label{eq:kaki}
\big|\dif_x^k\dif_{\xi}^l \Phi^{-1}\big|\leq CA^{k+l}(k+l)!(1+(k+l)^{s-1}\xim^{-\delta})^{k+l} \Phi^{-1}\xim^{-l+\delta k}.
\end{equation}
Indeed we have
\begin{align*}
|\dif_x^k\dif_{\xi}^l\Phi|\leq CA^{k+l}(k+l)!(1+(k+l)^{s-1}\xim^{-\delta})^{k+l}\big(b(s)\phi\xim^{-l+\delta k}\\
+\xim^{-2\delta}\xim^{-l}\big)\leq CA^{k+l}(k+l)!(1+(k+l)^{s-1}\xim^{-\delta})^{k+l} \Phi\xim^{-l+\delta k}.
\end{align*}
Then from $\Phi^{-1}\Phi=1$ one obtains the assertion for $\Phi^{-1}$ (see the proof of \cite[Lemma 4.1]{N:ojm}. Therefore we see
\begin{align*}
\big|\dif_x^k\dif_{\xi}^l\frac{|b'(s)|\phi }{\Phi}\big|\leq |b'(s)|CA^{k+l}(k+l)!(1+(k+l)^{s-1}\xim^{-\delta})^{k+l} \frac{\phi}{\Phi}\xim^{-l+\delta k}\\
=CA^{k+l}(k+l)!(1+(k+l)^{s-1}\xim^{-\delta})^{k+l} \frac{|b'(s)|\phi}{\Phi}\xim^{-l+\delta k}
\end{align*}
which proves the assertion.
\end{proof}
\begin{lemma}
\label{lem:ano:2} There are $A>0, C>0$ such that
\begin{align*}
\big|\dif_x^k\dif_{\xi}^l\dif_t\La(t,x,\xi)|\leq CA^{k+l}(k+l)!(1+(k+l)^{s-1}
 \xim^{-\delta})^{k+l}
 \dif_t\La(t,x,\xi) \xim^{-l+\delta k}
\end{align*}
for any $k,l\in\N$.
\end{lemma}
Taking Lemma \ref{lem:ano:1} into account  we study $e^{\psi}$ with $\psi$ satisfying 
\begin{equation}
\label{eq:adati}
\begin{split}
\big|\dif_x^k\dif_{\xi}^l\psi(x,\xi)|\leq C^{k+l+1}(k+l)!(1+(k+l)^{s-1}\xim^{-\delta})^{k+l}
 \lam(x,\xi) \xim^{-l+\delta k}
\end{split}
\end{equation}
for any $k,l\in \N$ with some  $\lam(x,\xi)>0$ and $C>0$. Define  $\om^{i}_{j}$ ($i,j\in\N$) by
\[
\dif_{\xi}^{i}\dif_x^{j}e^{\psi(x,\xi)}=\om^{i}_{j}(x,\xi)e^{\psi(x,\xi)}.
\]
\begin{lemma}There exist $B>0, C>0, A>0$ such that one has 
\begin{align*}
&|(\om^{i}_{j})^{(k)}_{(l)}|\leq CA^{i+j+k+l}\xim^{-(i+k)+\delta(j+l)}\\
\times \sum_{p=0}^{i+j-1}&(B\lam)^{i+j-p}(k+l+p)!(1+(k+l+p)^{s-1}\xim^{-\delta})^{k+l+p}
\end{align*}
for any $i,j, k, l\in \N$ where $(\om^i_j)^{(k)}_{(l)}=\dif_{\xi}^k\dif_x^l\om^i_j$.
\end{lemma}
\begin{proof}We first note that one can find $C>0, B>0$ such that for $i+j=1$
\begin{equation}
\label{jositu}
\begin{split}
|\dif_x^{l+j}\dif_{\xi}^{k+i}\psi|\leq C^{k+l}(k+l)!(1+(k+l)^{s-1}\xim^{-\delta})^{k+l}
 (B\lam)\xim^{-(k+i)+\delta( l+j)}
\end{split}
\end{equation}
for any $k, l\in \N$. Indeed from \eqref{eq:adati} 
\begin{align*}
|\dif_x^{l+j}\dif_{\xi}^{k+i}\psi|\leq C^{k+l+1}(k+l+1)!(1+(k+l+1)^{s-1}\xim^{-\delta})^{k+l+1}\\
\times \lam\,\xim^{-(k+i)+\delta (l+j)}\\
\leq C^{k+l+1}C_1^{k+l+1}(k+l)!(1+(k+l)^{s-1}\xim^{-\delta})^{k+l}
 \lam\,\xim^{-(k+i)+\delta (l+j)}
\end{align*}
with some $C_1>0$. Thus replacing $C$ by $CC_1$ and choosing $B=CC_1$ we get  \eqref{jositu}.
Assume that there exist $C>0, A_1>0, A_2>0$ such that for $i+j\leq n$ we have for any $k, l\in\N$
\begin{equation}
\label{eq:kasitu}
\begin{split}
&|(\om^{i}_{j})^{(k)}_{(l)}|\leq CA_2^{i+j}A_1^{k+l}\xim^{-(i+k)+\delta(j+l)}\\
\times \sum_{p=0}^{i+j-1}&(B\lam)^{i+j-p}(k+l+p)!(1+(k+l+p)^{s-1}\xim^{-\delta})^{k+l+p}.
\end{split}
\end{equation}
When $i+j=1$ the estimate \eqref{eq:kasitu} holds thanks to \eqref{jositu} for any $k,l\in\N$ because $\om^{i}_{j}=\dif_{\xi}^{i}\dif_x^{j}\psi$.  
When $e_1+e_2=1$ we have 
$\om^{i+e_1}_{j+e_2}=\om^{i(e_1)}_{j(e_2)}+
\psi^{(e_1)}_{(e_2)}\om^{i}_{j}$ then it follows that
\begin{align*}
|\om^{i+e_1(k)}_{j+e_2(l)}|=\Bigl|\om^{i(k+e_1)}_{j(l+e_2)}
+\sum\binom k{k'}\binom l{l'}\psi^{(e_1+k')}_{(e_2+l')}
\om^{i(k-k')}_{j(l-l')}\Bigr|\\
\leq CA_1^{k+l+1}A_2^{i+j}\xim^{-(i+k+e_1)+\delta(j+l+e_2)}
\sum_{p=0}^{i+j-1}(B\lam)^{i+j-p}\\
\times (k+l+1+p)! (1+(k+l+1+p)^{s-1}\xim^{-\delta})^{k+l+1+p}\\
+\sum\binom k{k'}\binom l{l'}(B\lam)C^{k'+l'+1}
(k'+l')!(1+(k'+l')^{s-1}\xim^{-\delta})^{k'+l'}\\
\times \xim^{-(k'+e_1)+\delta(l'+e_2)}
 CA_1^{k+l-k'-l'}A_2^{i+j}
\xim^{-(i+k-k')+\delta(j+l-l')}\\
\times \sum_{p=0}^{i+j-1}(B\lam)^{i+j-p}
 (k-k'+l-l'+p)!\\
\times (1+(k-k'+l-l'+p)^{s-1}\xim^{-\delta})^{k-k!+l-l'+p}.
\end{align*}
We denote the right-hand side by $\xim^{-(i+k+e_1)+\delta(j+l+e_2)}(I_1+I_2)$ where
\begin{align*}
I_1=CA_1^{k+l+1}A_2^{i+j}
\sum_{p=0}^{i+j-1}\sum_{i+j=p}(B\lam)^{i+j-p}\\
\times (k+l+1+p)! (1+(k+l+1+p)^{s-1}\xim^{-\delta})^{k+l+1+p}
\end{align*}
and the remaining part called $I_2$. We consider the sum over $k', l'$ in $I_2$:
\begin{align*}
&\sum\binom k{k'}\binom l{l'}C^{k'+l'+1}
(k'+l')!(1+(k'+l')^{s-1}\xim^{-\delta})^{k'+l'}A_1^{k+l-k'-l'}\\
&\times \sum_{p=0}^{\al+\be-1}
 (k-k'+l-l'+p)!(1+(k-k'+l-l'+p)^{s-1}\xim^{-\delta})^{k-k'+l-l'+p}
\end{align*}
which is bounded by
\begin{equation}
\label{eq:nitu}
\begin{split}
CA_1^{k+l}\sum_{k',l'}C^{k'+l'}A_1^{-k'-l'}2^{k'+l'}(1+(k'+l')^{s-1}\xim^{-\delta})^{k'+l'}\\
\times (1+(k-k'+l-l'+p)^{s-1}\xim^{-\delta})^{k-k'+l-l'+p}
\end{split}
\end{equation}
because 
\[
\binom k{k'}\binom l{l'}
(k'+l')!(k-k'+l-l'+p)!\leq 2^{k'+l'}(k+l+p)!.
\]
Note that
\begin{align*}
(1+(k'+l')^{s-1}\xim^{-\delta})^{k'+l'}
(1+(k-k'+l-l'+p)^{s-1}\xim^{-\delta})^{k-k'+l-l'+p}\\
\leq (1+(k+l+p)^{s-1}\xim^{-\delta})^{k'+l'}
(1+(k+l+p)^{s-1}\xim^{-\delta})^{k-k'+l-l'+p}\\
\leq (1+(k+l+p)^{s-1}\xim^{-\delta})^{k+l+p}
\end{align*}
then \eqref{eq:nitu} is estimated by
\begin{align*}
CA_1^{k+l}(1+(k+l+p)^{s-1}\xim^{-\delta})^{k+l+p}\sum_{k'=0}^k\big(2C/A_1\big)^{k'}\sum_{l'=0}^l\big(2C/A_1\big)^{l'}\\
\leq CA_1^{k+l+2}(A_1-2C)^{-2}(1+(k+l+p)^{s-1}\xim^{-\delta})^{k+l+p}.
\end{align*}
Thus we obtain
\begin{align*}
I_2\leq C^2A_2^{i+j}A_1^{k+l+2}(A_1-2C)^{-2}\sum_{p=0}^{i+j-1}(B\lam)^{i+j+1-p}\\
\times (k+l+p)!(1+(k+l+p)^{s-1}\xim^{-\delta})^{k+l+p},\\
I_1\leq CA_1^{k+l+1}A_2^{i+j}
\sum_{p=1}^{i+j}(B\lam)^{i+j+1-p}
 (k+l+p)!\\
\times  (1+(k+l+p)^{s-1}\xim^{-\delta})^{k+l+p}.
\end{align*}
Therefore we have proved that
\begin{align*}
&I_1+I_2\leq CA_2^{i+j+1}A_1^{k+l}\big(A_1A_2^{-1}+CA_1^2A_2^{-1}(A_1-2C)^{-2}\big)\\
&\times \sum_{p=0}^{i+j}
(B\lam)^{i+j+1-p}
 (k+l+p)! (1+(k+l+p)^{s-1}\xim^{-\delta})^{k+l+p}.
\end{align*}
Choosing $A_1, A_2$ such that $
A_1A_2^{-1}+CA_1^2A_2^{-1}(A_1-2C)^{-2}\leq 1$ 
the estimate \eqref{eq:kasitu} holds for $i+j=n+1$. 
\end{proof}
Noting that $\sum_{p=0}^{i+j-1}(B\lam)^{i+j-p}(k+l+p)!(1+(k+l+p)^{s-1}\xim^{-\delta})^{k+l+p}$ is bounded by
\begin{align*}
C_1^{k+l}B\lam\,(k+l)!(1+(k+l)^{s-1}\xim^{-\delta})^{k+l}
 \sum_{p=0}^{i+j-1}C_1^p\big(B\lam\big)^{i+j-1-p}(p+p^s\xim^{-\delta}\big)^p\\
\leq C_2^{k+l+i+j}B\lam\,(k+l)!(1+(k+l)^{s-1}\xim^{-\delta})^{k+l}\\
\times (B\lam+(i+j-1)+(i+j-1)^s\xim^{-\delta})^{i+j-1}
\end{align*}
if $i+j\geq 1$ we have the following lemma (see \cite[Lemma 5.1]{N:ojm}):
\begin{lemma}
\label{lem:seikei}
Assume that $\psi$ satisfies \eqref{eq:adati}. Then we have
\begin{equation}
\label{eq:hiei}
|\dif_{\xi}^{i}\dif_x^{j}e^{\psi}|\leq
C_2^{i+j}(B\lam+i+j
+(i+j)^{s}\xim^{-\delta})^{i+j} \xim^{-i+\delta j}e^{\psi}
\end{equation}
and $(i+j\geq 1)$
\begin{equation}
\label{eq:hira}
\begin{split}
|\dif_{\xi}^k\dif_x^l(\om^{i}_{j})| \leq
 C^{i+j+k+l}B\lam\,(B\lam+(i+j-1)
 +(i+j-1)^s\xim^{-\delta})^{i+j-1}\\
 \times 
 (k+l)!(1+(k+l)^{s-1}\xim^{-\delta})^{k+l}\xim^{-(i+k)+\delta(j+l)}.
\end{split}
\end{equation}
\end{lemma}
We apply these results to $\psi=\pm\Lambda$ and $\om^{i}_{j}=T^{i}_{j}$ to obatain
\begin{corollary}
\label{cor:moriyama}There exist $A>0, C>0$ such that
\begin{equation}
\label{eq:Mgsymbol:a}
\begin{split}
|\dif_{\xi}^k\dif_x^le^{\pm\Lambda}|\leq
CA^{k+l}(\xim^{\tika}+k+l+(k+l)^{s}\xim^{-\delta})^{k+l} \xim^{-k+\delta l}e^{\pm\Lambda}
\end{split}
\end{equation}
and for any $i, j$ $(i+j\geq 1)$
\begin{equation}
\label{eq:Mgsymbol:b}
\begin{split}
|\dif_{\xi}^k\dif_x^lT^{i}_{j}|\leq CA^{k+l}(k+l)!(1+(k+l)^{s-1}\xim^{-\delta})^{k+l}\\
\times \xim^{\tika}(\xim^{\tika}+(i+j-1)+(i+j-1)^s\xim^{-\delta})^{i+j-1}
 \xim^{-(i+k)+\delta(j+l)}.
\end{split}
\end{equation}
\end{corollary}
%
We introduce some symbol classes \cite[Definition 4.1]{N:ojm} (see also \cite{N1}):
\begin{definition}\rm
\label{dfn:kurasu:b}Let $0\leq \delta<1$ and $1<s$. Let $m(x,\xi;\mu)$ be a positive function with a small parameter $\mu>0$. We say that $a(x,\xi;\mu)\in C^{\infty}(\R\times\R ) $ 
belongs to  ${S}^{\lr{s}}_{\delta}(m)$ if for any $\varep>0$ there exist $C>0, A>0$ independent of 
$0<\mu\le 1$ such that for all $i, j\in\N$ one has
\begin{align*}
\big|
\partial_x^{j}\partial_{\xi}^{i}a(x,\xi;\mu)
\big|
\leq
C\,A^{i+j}\ 
((i+j)^{1+\varep}
+(i+j)^s\xim^{-\delta})^{i+j}\xim^{-i+j\delta}m(x,\xi,\mu)
\,.
\end{align*}
\end{definition}
\begin{definition}\rm
Let $s>1$ and $m(x,\xi,\mu)$ be a positive function with a small parameter $\mu>0$. We say that $a(x,\xi,\mu)\in C^{\infty}(\R\times\R)$ belongs to $S_{0,0}^{(s)}(m)$ if there exist $C>0, A>0$ independent of $\mu$ such that one has
\[
|\dif_x^j\dif_{\xi}^ia(x,\xi,\mu)|\leq CA^{i+j}(i+j)^{s(i+j)}m(x,\xi,\mu)
\]
for any $i,j\in \N$.
\end{definition}
We often write $a(x,\xi)$ for $a(x,\xi;\mu)$ dropping $\mu$.  Now we introduce the metric (see \cite{CN2})
\[
g(dx,d\xi)=g_{x,\xi}(dx,d\xi)=\phi(x,\xi)^{-1}dx^2+\xim^{-2}d\xi^2
\]
and $g_{\delta}(dx,d\xi)=\xim^{2\delta}dx^2+\xim^{-2}d\xi^2$. It is clear that $
g\leq g_{\delta}$ 
and $g_{\delta}$ is the metric defining the class $S_{1,\delta}$ for any fixed $\mu>0$.  
\begin{lemma}
Assume $0<\delta<1$. Then the metric $g$ is slowly varying and $\sigma$ temperate $($uniformly in $\mu$$)$ and 
\begin{equation}
\label{eq:hikoshi}
h^2=\sup g/g^{\sigma}=\xim^{-2}\phi^{-1}\leq \xim^{2\delta-2}=\xim^{ -2(1-\delta)}\leq 1
\end{equation}
that is, $g$ is an admissible metric.
\end{lemma}
\begin{lemma}
\label{lem:phiOK} $\phi^{\pm 1}$ are $g$ continuous and $\sigma, g$ temperate $($uniformy in $\mu$$)$, that is $g$-admissible weights.
\end{lemma}
\begin{proof}
Proofs of the above two lemmas are found in \cite{CN2}.
\end{proof}
\begin{definition}\rm
\label{dfn:kurasu}
We denote $S_{\phi}(m)=S(m,g)$ and $S_{\delta}(m)=S(m,g_{\delta})$ for a $g$-admissible and $g_{\delta}$-admissible $($uniformly in $\mu$$)$ weight $m(x,\xi,\mu)>0$ respectively. By $a\in \mu^{c}S_{\phi}(m)$ (resp. $a\in \mu^c S_{\delta}(m)$) we mean that $\mu^{-c}a\in  S_{\phi}(m)$ (resp. $\mu^{-c}a\in S_{\delta}(m)$) uniformly in small $\mu>0$. 
\end{definition}
\begin{lemma}
\label{lem:ano:1bis}For any $k,l\in \N$ there is $C>0$ such that
\[
\big|\dif_x^k\dif_{\xi}^l\La(t,x,\xi)|\leq C\Lambda \phi^{-k/2}\xim^{-l}.
\]
\end{lemma}
\begin{proof} Note that with $\Phi=b(s)\phi(x,\xi)+\xim^{-2\delta}$ one has
\begin{equation}
\label{eq:kakibis}
\big|\dif_x^k\dif_{\xi}^l \Phi(s,x,\xi)^{-1}\big|\leq C_{kl} \Phi(s,x,\xi)^{-1}\phi^{-k/2}\xim^{-l}.
\end{equation}
Indeed we have
\begin{align*}
|\dif_x^k\dif_{\xi}^l\Phi(s,x,\xi)|\leq C \big(b(s)\phi(x,\xi)\phi^{-k/2}\xim^{-l}
+\xim^{-2\delta}\xim^{-l}\big)\\\leq C\Phi(s,x,\xi)\phi^{-k/2}\xim^{-l}.
\end{align*}
Then from $\Phi^{-1}\Phi=1$ one obtains the assertion for $\Phi^{-1}$. Therefore we see
\begin{equation}
\label{eq:kaki:a}
\Big|\dif_x^k\dif_{\xi}^l\frac{|b'(s)|\phi(x,\xi)}{\Phi(s,x,\xi)}\Big|\leq C\frac{|b'(s)|\phi(x,\xi)}{\Phi(s,x,\xi)}\phi^{-k/2}\xim^{-l}
\end{equation}
which proves the assertion.
\end{proof}
\begin{lemma}
\label{lem:katto:bis}
For any $k, l\in \N$ there exist $C>0$ such that
\[
|\dif_x^{k}\dif_{\xi}^{l}\phi(x,\xi)|/\phi\leq C\phi^{-k/2}\xim^{-l}
\]
that is $\phi\in S_{\phi}(\phi)$.
\end{lemma}
\begin{proof}
Note that $|\dif_xa(x)|\leq C\sqrt{a}\leq C\phi^{1/2}$ since $a(x)\geq 0$  and for $k\geq 2$
\[
|\dif_x^ka(x)|\leq C_k\phi^{1-k/2}
\]
since $\phi^{-1}\geq c$ with some $c>0$. Taking \eqref{eq:sanada} into account the proof is clear.
\end{proof}
\begin{lemma}
\label{lem:ano:2bis} For any $k,l\in\N$ there is $C>0$ such that
\[
\big|\dif_x^k\dif_{\xi}^l\dif_t\La(t,x,\xi)|\leq C\dif_t\La(t,x,\xi) \phi^{-k/2}\xim^{-l}.
\]
\end{lemma}
\begin{lemma}
\label{lem:moriyama:bis}For any $i, j$, $k,l$  there exists $C
>0$ such that
\begin{equation}
\label{eq:Mgsymbol:bbis}
|\dif_{\xi}^k\dif_x^lT^{i}_{j}|\leq C\xim^{\tika(i+j)}\phi^{-(j+l)/2}\xim^{-(i+k)}
\end{equation}
that is $T^i_j\in S_{\phi}(\xim^{\tika(i+j)}\phi^{-j/2}\xim^{-i})$.
\end{lemma}
Since $\phi^{-1}\leq \xim^{2\delta}$ we have $S_{\phi}(\xim^m)\subset S_{\delta}(\xim^m)$.
\begin{lemma}
\label{lem:budo:a}We have
$\Lambda \in S_{\phi}(\xim^{\tika})\cap S_{\delta}^{\lr{s}}(\xim^{\tika})$ and $\dif_t\Lambda \in S_{\phi}(\sqrt{\phi}\xim^{\delta})\cap S_{\delta}^{\lr{s}}(\xim^{\delta})$. 
\end{lemma}
\begin{proof}The first assertion follows from Lemmas \ref{lem:ano:1}, \ref{lem:ano:1bis} and Corollary \ref{cor:kurinasi}. Noting  
\[
\dif_t\La(t,x,\xi)=\phi^{1/2}\frac{|b'(t)|\phi^{1/2}}{b(t)\phi+\xim^{-2\delta}}\leq \phi^{1/2}\xim^{\delta}
\]
the second assertion follows from Lemmas \ref{lem:ano:2} and \ref{lem:ano:2bis}.
\end{proof}
\begin{lemma}
\label{lem:yoritune}
We have $e^{-\Lambda}\in S_{\delta}^{\lr{s}}(1)\subset S_{\delta}(1)$.
\end{lemma}
\begin{proof}
Thanks to Lemmas \ref{lem:ano:1} and \ref{lem:seikei} one has
\begin{align*}
|\dif_{\xi}^i\dif_x^je^{-\Lambda}|\leq
CA^{i+j}(\Lambda+(i+j)+(i+j)^{s}\xim^{-\delta})^{i+j} \xim^{-i+\delta j}e^{-\Lambda}.
\end{align*}
Since $\Lambda\geq 0$ one has $\Lambda^{i+j}e^{-\Lambda}\leq A_1^{i+j}(i+j)^{i+j}$ and hence the right-hand side is bounded by
\[
CA^{i+j}((i+j)+(i+j)^{s}\xim^{-\delta})^{i+j} \xim^{-i+\delta j}
\]
which proves the assertion.
\end{proof}
\begin{lemma}
\label{lem:piko}Let $c>0$ and $0<\kappa<1$. Then for any $0<c'<c$ and $s>1$, $\delta>0$  one has 
$e^{-c\xim^{\kappa}}\in S_{\delta}^{\lr{s}}(e^{-c'\xim^{\kappa}})$.
\end{lemma}
\begin{proof}
Since $|\dif_{\xi}^l\xim^{\kappa}|\leq C^{l+1}l!\xim^{\kappa}\xim^{-l}$ then from Lemma \ref{lem:seikei} there exists $C_1>0$ such that
\[
|\dif_{\xi}^le^{-c\xim^{\kappa}}|\leq C_1^{l+1}(\xim^{\kappa}+l)^l\xim^{-l}e^{-c\xim^{\kappa}}.
\]
For any $c'<c$ there is $A>0$ such that $\xim^{\kappa l}e^{-(c-c')\xim^{\kappa}}\leq A^{l+1}l!$ and hence we have
\[
|\dif_{\xi}^le^{-c\xim^{\kappa}}|\leq C_1^{l+1}l^l\xim^{-l}e^{-c'\xim^{\kappa}}
\]
which proves the assertion.
\end{proof}
%

\section{Composition of PDO's acting in the Gevrey classes }

We recall several facts on compositions of pseudodifferential operators including ${\rm Op}(e^{\pm \Lambda})$  whose proofs are given in Appendix. We denote by ${\rm Op}(a)$ the Weyl quantization of $a(x,\xi)$  (see \cite{Ho}). Here we recall
\[
\delta=1-\kappa,\quad \kappa>\tika,\quad s\kappa<1.
\]
\begin{proposition}
\label{pro:weyl:1} 
Assume 
$a(x)\in {\mathcal G}_0^s(\R)$ 
which is constant for large $|x|$. Then  
the operator $b(x,D)=e^{\tau \mD^{\kappa}}a(x)e^{-\tau\mD^{\kappa}}$ is a pseudodifferential operator with symbol  given by
\[
b(x,\xi) =a(x) -i\tau
\dif_xa(x)
\dif_{\xi}\xim^{\kappa}+q(x,\xi)
+r(x,\xi)
\]
with $q(x,\xi)\in S_{\delta}^{\lr{s}}(\xim^{-2+2\kappa})$ and 
$r(x,\xi)\in S_{0,0}^{(s)}(e^{-c\xim^{1/s}})$ with some $c>0$ where  $\tau$ is a real parameter.
\end{proposition}
\begin{proposition}
\label{pro:comLam}Assume $b(x,\xi)\in S^{\lr{s}}_{\delta}(\xim^m)$. Then we have
\begin{equation}
\label{eq:comLam}
\big(be^{\pm \Lambda}\big)\#e^{\mp\Lambda}=b+b^{\pm}+R^{\pm}
\end{equation}
where $b^{\pm}\in S^{\lr{s}}_{\delta}\big(\xim^{m-(\kappa-{\tilde\kappa})}\big)$ and $R^{\pm}\in S^{({\bar s})}_{0,0}(e^{-c\xim^{\baka}})$ with some $c>0$, ${\bar s}$ and $\baka$ such that $\bas \tika<1$ and $\baka>\tika$.
\end{proposition}
\begin{corollary}
\label{cor:kaneko}
There is $\tis>1$ such that
\[
e^{\pm\La}\#e^{\mp \La}-1\in  S^{\lr{\tis}}_{\delta}(\xim^{-(\kappa-\tika)})\subset \mu^{\kappa-\tika} S_{\delta}(1).
\]
\end{corollary}
\begin{corollary}
\label{cor:watabe}Assume $b(x,\xi)\in S^{\lr{s}}_{\delta}(\xim^m)$ then 
\[
\big(be^{\pm \Lambda}\big)\#e^{\mp\Lambda}\in S_{\delta}(\xim^m).
\]
\end{corollary}
\begin{proof}
Since $S_{0,0}^{(\bas)}(e^{-c\xim^{\baka}})\subset S_{\delta}(\xim^k)$ for any $k$ the proof is immediate. 
\end{proof}
\begin{proposition}
\label{pro:abLam}
Assume $b(x,\xi)\in S^{\lr{s}}_{\delta}(\xim^{m})$ and $a(x,\xi)\in S^{\lr{s}}_{\delta}(\xim^d)$ then 
\begin{eqnarray*}
\big(be^{-\Lambda}\big)\#a=\sum_{k+l< N}\frac{(-1)^k}{(2i)^{k+l}k!l!}\big(\dif_x^k\dif_{\xi}^l(be^{-\Lambda})\big)\dif_x^{l}\dif_{\xi}^{k}a+q_Ne^{-\Lambda}+R_N,\\
a\#\big(be^{-\Lambda}\big)=\sum_{k+l<N}\frac{(-1)^k}{(2i)^{k+l}k!l!}\dif_x^k\dif_{\xi}^la\big(\dif_x^l\dif_{\xi}^k(be^{-\Lambda})\big)+{\tilde q}_Ne^{-\Lambda}+{\tilde R}_N
\end{eqnarray*}
with $q, {\tilde q}\in S^{\lr{s}}_{\delta}(\xim^{m+d-(\kappa-\tika)N})$ and $R_N, {\tilde R}_N\in S^{(\bas)}_{0,0}(e^{-c\xim^{\baka}})$  with some ${\bar s}>1$, $\baka>0$ and $c>0$ satisfying  $\bas \tika<1$ and $\baka>\tika$. 
\end{proposition}
\begin{corollary}
\label{cor:nao} 
Assume $b(x,\xi)\in S^{\lr{s}}_{\delta}(\xim^{m})$ and $a(x,\xi)\in S^{\lr{s}}_{\delta}(\xim^d)$ then   
\begin{eqnarray*}
\big(be^{-\Lambda}\big)\#a=c\,e^{-\Lambda}+r,\quad
a\#\big(be^{-\Lambda}\big)={\tilde c}\,e^{-\Lambda}+{\tilde r}
\end{eqnarray*}
with $c, {\tilde c}\in S^{\lr{s}}_{\delta}(\xim^{m+d})$ and $r, {\tilde r}\in S^{(\bar s)}_{0,0}(e^{-c\xim^{\baka}})$ with some ${\bar s}>1$, $\baka>0$ and  $c>0$ satisfying $\bas \tika <1$ and $\baka>\tika$.
\end{corollary}
\begin{proof}
Write
\[
\big(\dif_x^k\dif_{\xi}^l(be^{-\Lambda})\big)\dif_x^{l}\dif_{\xi}^{k}a=c_{k,l}e^{-\Lambda}.
\]
Then to prove the assertion it suffices to note $c_{k,l}\in S^{\lr{s}}_{\delta}(\xim^{m+d})$ which can be seen from Corollary \ref{cor:moriyama} because $\kappa>\tika$.
\end{proof}
\begin{proposition}
\label{pro:kodaira}
Assume $p\in S^{(\bar s)}_{0,0}(e^{-c\xim^{\baka}})$ with $\bas \tika<1$, $\baka>\tika$ and $c>0$. Then 
\[
p\#e^{ \Lambda},\quad  e^{ \Lambda}\#p\in S^{\lr{s^*}}_{\delta}\big(e^{-c\xim^{\kappa^*}}\big)
\]
with some $s^*>1$ and $\kappa^*>\tika$. In particular we have $p\#e^{ \Lambda}$, $e^{ \Lambda}\#p\in S_{\delta}(\xim^l)$ for amy $l\in\R$.
\end{proposition}
\begin{corollary}
\label{cor:takagi}
Assume $b(x,\xi)\in S^{\lr{s}}_{\delta}(\xim^{m})$ and $a(x,\xi)\in S^{\lr{s}}_{\delta}(\xim^d)$ then  
\begin{eqnarray*}
\big(be^{-\Lambda}\big)\#a\#e^{\Lambda},\quad
e^{\Lambda}\#a\#\big(be^{-\Lambda}\big)\in S_{\delta}(\xim^{m+d}).
\end{eqnarray*}
\end{corollary}
\begin{proof}
Thanks to Corollary \ref{cor:nao} one can write $
\big(be^{-\Lambda})\#a=c\,e^{-\Lambda}+r$ 
with $c\in S_{\delta}^{\lr{s}}(\xim^{m+d})$ and $r\in S_{0,0}^{(\bas)}(e^{-c\xim^{\baka}})$ where $\bas \tika<1$, $\baka>\tika$ and $c>0$. Write
\[
\big(be^{-\Lambda}\big)\#a\#e^{\Lambda}=\big(c\,e^{-\Lambda}\big)\#e^{\Lambda}+r\#e^{\Lambda}
\]
then in virtue of Corollary \ref{cor:watabe} one has $(c\,e^{-\Lambda})\#e^{\Lambda}\in S_{\delta}(\xim^{m+d})$.  On the other hand from Proposition \ref{pro:kodaira} it follows that $r\#e^{\Lambda}\in S_{\delta}(\xim^{m+d})$ and hence the assertion.
\end{proof}
\begin{lemma}
\label{leiva}
Assume $p\in S^{\lr{s}}_{\delta}(\xim^{m})$ then there is  $C>0$ such that
\[
\|{\rm Op}(p e^{-\Lambda})u\|\leq C\|\mD^{m}{\rm Op}(e^{-\La}) u\|.
\]
\end{lemma}
\begin{proof} From Corollary \ref{cor:kaneko} it follows that
\[
e^{-\La}\#e^{\La}=1-R,\quad e^{\La}\#e^{-\La}=1-{\tilde R}
\]
with $R$, ${\tilde R}\in \mu^{\kappa-\tika}S_{\delta}(1)$. Choosing $\mu>0$ small the inverse $(1-{\rm Op} (R))^{-1}$ and $(1-{\rm Op}({\tilde R}))^{-1}$ are well-defined as a bounded operator on $L^2$. Thanks to \cite{Be} (see also \cite{Ler}) there exist $K$ and ${\tilde K}\in S^0_{1, \delta}$ such that 
${\rm Op}(K)=(1-{\rm Op}(R))^{-1}$ and ${\rm Op}({\tilde K})=(1-{\rm Op}({\tilde R}))^{-1}$. Therefore one has 
 \begin{equation}
 \label{eqivu}
{\rm Op}( K){\rm Op}(e^{-\La}){\rm Op}(e^{\La})=1,\quad {\rm Op}(e^{\La}){\rm Op}(e^{-\La}){\rm Op}({\tilde K})=1.
 \end{equation}
 Remark that 
 \begin{equation}
 \label{eq:ofuda}
 {\rm Op}(e^{\La}){\rm Op}(K){\rm Op}(e^{-\La})=1, \quad {\rm Op}(e^{-\Lambda}){\rm Op}({\tilde K}){\rm Op}(e^{\Lambda})=1.
 \end{equation}
Set $(pe^{-\Lambda})\#e^{\La}=q$ which belongs to $S_{\delta}(\xim^m)$ by Corollary \ref{cor:watabe}.  Note that $
{\rm Op}(pe^{-\Lambda})={\rm Op}(pe^{-\Lambda}){\rm Op}(e^{\La}){\rm Op}(K){\rm Op}(e^{-\La})={\rm Op}(q){\rm Op}(K){\rm Op}(e^{-\La})$ 
and hence
\[
\|{\rm Op}(pe^{-\Lambda})u\|=\|{\rm Op}(q){\rm Op}(K)\mD^{-m}\mD^{m}{\rm Op}(e^{-\La})u\|\leq C\|\mD^m {\rm Op}(e^{-\La})u\|
\]
since $q\#K\#\xim^{-m}\in S_{\delta}(1)$. This proves the assertion. 
\end{proof}
\begin{corollary}
\label{cor:koike}
For any $k,l\in\R$ there is $C>0$ such that 
\[
\|{\rm Op}(e^{-\Lambda})\mD^{k+l}u\|/C\leq\|\mD^{k}{\rm Op}(e^{-\Lambda})\mD^lu\|\leq C\|\mD^{k+l} {\rm Op}(e^{-\Lambda})u\|.
\]
\end{corollary}
\begin{proof}
Using \eqref{eq:ofuda} write 
\[
\mD^k{\rm Op}(e^{-\Lambda})\mD^l=\mD^k\big({\rm Op}(e^{-\Lambda})\mD^l{\rm Op}(e^{\Lambda})\big){\rm Op}(K){\rm Op}(e^{-\Lambda}).
\]
Thanks to Corollary \ref{cor:takagi} we see $e^{-\Lambda}\#\xim^l\#e^{\Lambda}\in S_{\delta}(\xim^l)$ and hence 
\[
\xim^k\#\big(  e^{-\Lambda}\#\xim^l\#e^{\Lambda}\big)\#K\in S_{\delta}(\xim^{k+l}).
\]
This shows the second inequality. The first inequality follows from the second one.
\end{proof}
\begin{lemma}
\label{lem:sanuki}
Assume $R\in S^{(\bar s)}_{0,0}(e^{-c\xim^{\baka}})$ with $\bas \tika<1$, $\baka>\tika$ and $c>0$ and $p\in S_{\delta}^{\lr{s}}(\xim^m)$. Then for any $k, l\in\R$ there is $C$ such that
\[
\|\mD^k{\rm Op}(R){\rm Op}(p)u\|\leq C\mu \|\mD^l{\rm Op}(e^{-\Lambda})u\|.
\]
\end{lemma}
\begin{proof}We first show that for any $k, l\in\R$ we have 
\begin{equation}
\label{eq:suri}
\|\mD^k{\rm Op}(R)u\|\leq C\mu \|\mD ^{l}{\rm Op}(e^{-\Lambda})u\|.
\end{equation}
Indeed using \eqref{eq:ofuda} we write ${\rm Op}(R)={\rm Op}(R){\rm Op}(e^{\La}){\rm Op}(K){\rm Op}(e^{-\La})$. Thanks to Proposition \ref{pro:kodaira} it follows that $R\#e^{\Lambda}\in S_{\delta}(\xim^{l-k-1})$ for any $l\in\R$. Therefore one has $\xim^k\#(R\#e^{\Lambda})\in S_{\delta}(\xim^{l-1})$. Then $\xim^k\#(R\#e^{\Lambda})\#K\in S_{\delta}(\xim^{l-1})\subset \mu S_{\delta}(\xim^{l})$ for any $l$ and hence \eqref{eq:suri}. Therefore one has
\[
\|\mD^k{\rm Op}(R){\rm Op}(p)u\|\leq C\mu \|\mD ^{l-m}{\rm Op}(e^{-\Lambda}){\rm Op}(p)u\|.
\]
From Corollary \ref{cor:nao} one can write
\[
\xim^{l-m}\#(e^{-\Lambda}\#p)=\xim^{l-m}\#(qe^{-\Lambda})+\xim^{l-m}\#r
\]
where $q\in S_{\delta}^{\lr{s}}(\xim^m)$ and $r\in S_{0,0}^{(\bas)}(e^{-c\xim^{\baka}})$ with $\bas \tika<1$ and $\baka>\tika$. By Corollary \ref{cor:nao} again it follows that $\xim^{l-m}\#(qe^{-\Lambda})=c\,e^{-\Lambda}+R$ where $c\in S_{\delta}^{\lr{s}}(\xim^l)$ and $R\in S_{0,0}^{(\bas)}(e^{-c\xim^{\baka}})$ with possibly different $\bas$, $\baka$ satisfying the same conditions. Therefore we have
\[
\|\mD^{l-m}{\rm Op}(qe^{-\Lambda})u\|\leq C\|\mD^l {\rm Op}(e^{-\Lambda})u\|
\]
thanks to Lemma \ref{leiva}. Applying \eqref{eq:suri} again one obtains
\[
\|\mD^{l-m}{\rm Op}(r)u\|\leq C\mu \|\mD^l{\rm Op}(e^{-\Lambda})u\|
\]
and we conclude the assertion.
\end{proof}
\begin{lemma}
\label{lem:yodobasi}
Assume $p\in S_{\delta}^{\lr{s}}(\xim^{m_1})$ and $q\in S_{\delta}^{\lr{s}}(\xim^{m_2})$. Then for any $k\in\R$ there exists $C>0$ such that
\begin{align*}
\|\mD^k{\rm Op}(qe^{-\Lambda}){\rm Op}(p)u\|\leq C\|\mD^{k+m_1+m_2}{\rm Op}(e^{-\Lambda})u\|\\
\leq C\mu \|\mD^{k+m_1+m_2+1}{\rm Op}(e^{-\Lambda})u\|
\end{align*}
\end{lemma}
\begin{proof}
Thanks to Corollary \ref{cor:nao} one can write $
\xim^k\#(qe^{-\Lambda})=c\,e^{-\Lambda}+r$ 
where $c\in S_{\delta}^{\lr{s}}(\xim^{k+m_2})$ and $r\in S_{0,0}^{(\bas_1)}(e^{-c\xim^{{\bar\kappa}_1}}$ with ${\bar s}_1\tika<1$ and ${\bar\kappa}_1>\tika$. Thus we have
\[
\xim^k\#(qe^{-\Lambda})\#p=(c\,e^{-\Lambda})\#p+r\#p
\]
and Corollary \ref{cor:nao} proves that $(c\,e^{-\Lambda})\#p={\tilde c}\,e^{-\Lambda}+R$ where $
{\tilde c}\in S_{\delta}^{\lr{s}}(\xim^{k+m_1+m_2})$ and $R\in S_{0,0}^{(\bas_2)}(e^{-c\xim^{{\bar\kappa}_2}})$ 
 with ${\bar s}_2\tika<1$ and ${\bar\kappa}_2>\tika$. Since 
 \[
 \|{\rm Op}(r){\rm Op}(p)u\|\leq C\|\mD^{k+m_1+m_2}{\rm Op}(e^{-\Lambda})u\|
 \]
 in virtue of Lemma \ref{lem:sanuki} the proof follows from Lemmas \ref{leiva} and Lemma \ref{lem:sanuki}.
\end{proof}
 \begin{lemma}
 \label{lem:koreha} Assume $p\in S_{\phi}(m)\cap S_{\delta}^{\lr{s}}(\xim^l)$ where $m$ is $g$-admissible, $l\in\R$. Then for any $M \in \N$ one can write
 \begin{align*}
pe^{-\Lambda}=(p+p_1+p_2)\#e^{-\Lambda}+qe^{-\Lambda}+r,\\
 pe^{-\Lambda}=e^{-\Lambda}\#(p+{\tilde p}_1+{\tilde p}_2)+{\tilde q}e^{-\Lambda}+{\tilde r}
 \end{align*}
with $p_j, {\tilde p}_j\in  S_{\phi}(m \phi^{-j/2}\xim^{-j+j\tika})\cap S_{\delta}^{\lr{s}}(\xim^{l-(\kappa-\tika)j})$ where $p_1$, ${\tilde p}_1$ are pure imaginary and $q, {\tilde q}\in S_{\delta}^{\lr{s}}(\xim^{-M})$,  $r, {\tilde r}\in S_{0,0}^{\lr{\bas}}(e^{-c\xim^{\baka}})$ with some $\bas>1$, $\baka>0$ and $c>0$ such that $\bas \tika<1$ and $\baka>\tika$.
 \end{lemma}
\begin{proof} Thanks to Proposition \ref{pro:abLam} we can write $pe^{-\Lambda}=p\#e^{-\Lambda}+p_1e^{-\Lambda}+p_2e^{-\Lambda}+qe^{-\Lambda}+r$ where $q\in  S_{\delta}^{\lr{s}}(\xim^{-M})$ choosing $N$ such that $-l+(\kappa-\tika)N\geq M$  and  $r\in S_{0,0}^{(\bas)}(e^{-c\xim^{\baka}})$ with $\bas\tika<1$ and $\baka>\tika$. It follows from Lemma \ref{lem:moriyama:bis} that $p_j\in  S_{\phi}(m\phi^{-j/2}\xim^{j\tika-j})\cap S_{\delta}^{\lr{s}}(\xim^{l-(\kappa-\tika)j})$. Note that $p_1$ is pure imaginary. Repeating the same arguments we have $p_1e^{-\Lambda}=p_1\#e^{-\Lambda}+{\tilde p}_2e^{-\Lambda}+{\tilde q}e^{-\Lambda}+{\tilde r}$ so that
\[
pe^{-\Lambda}=(p+p_1)\#e^{-\Lambda}+p_2'e^{-\Lambda}+q'e^{-\Lambda}+r'
\]
where $p_2'\in  S_{\phi}(m\phi^{-1}\xim^{2\tika-2})\cap S_{\delta}^{\lr{s}}(\xim^{l-2(\kappa-\tika)})$ and $q'\in  S_{\delta}^{\lr{s}}(\xim^{-M})$ and $r'\in S_{0,0}^{(\bas)}(e^{-c\xim^{\baka}})$ with possibly different $\bas$, $\baka$, $c>0$ such that $\bas\tika<1$ and $\baka>\tika$. Repeating the same argument to $p_2'e^{-\Lambda}$ we conclude the assertion. The second assertion can be proved similarly.
\end{proof}
%


\section{Energy estimates}

Instead of $P=D_t^2-b(t)D_xa(x)D_x$  we study 
\[
Pu=D_t^2u-b(t)\mD a(x)\mD u
\]
which differs from ${\tilde P}$ only by a zero-th order term  which is irrelevant in the arguments proving Proposition \ref{pro:thva}. 
Assume $0\leq a(x)\in {\mathcal G}_0^s(\R)$ and consider
\[
P^{\sharp}={\rm Op}(e^{ (\tau-\theta t)\xim^{\kappa}})P\,{\rm Op}(e^{- (
\tau-\theta t)\xim^{\kappa}})
\]
where $\tau>0$ is a fixed positive constant and $\theta>0$ is a positive parameter where $0\leq \theta t\leq \tau$. Thanks to Proposition \ref{pro:weyl:1} we have
\begin{equation}
\label{eq:dedasi}
\begin{split}
P^{\sharp}=(D_t-i \theta \mD^{\kappa})^2-b(t)\mD a(x)\mD\\
-b(t)\mD {\rm Op}(a_1)\mD+R\\
=A^2-b(t)\mD a\mD -b(t)\mD{\rm Op}( a_1)\mD +R
\end{split}
\end{equation}
with $A=D_t-i \theta \mD^{\kappa}$ where $R=b(t)\mD{\rm Op}(q+r)\mD$ with $
q(x,\xi)\in S_{\delta}^{\lr{s}}(\xim^{-2+2\kappa})$ and $r(x,\xi)\in S_{0,0}^{(s)}(e^{-c\xim^{1/s}})$ with some $c>0$  
and 
\begin{eqnarray*}
a_1(t,x,\xi)= (\tau-\theta t)D_xa(x)\dif_{\xi}\xim^{\kappa}.
\end{eqnarray*}
In order to prove Theorem \ref{thm:taifu} we derive an apriori estimate for $P^{\sharp}$. We now introduce the energy:
\begin{align*}
E(u)=\|{\rm Op}(e^{-\La})Au\|^2+{\mathsf{ Re}}\,(b(t)a(\cdot){\rm Op}(e^{-\La})\mD u,{\rm Op}(e^{-\La})\mD
u)\\
+\|\mD^{\kappa} {\rm Op}(e^{-\La}) u\|^2
\end{align*}
with $\La(t,x,\xi)$ defined by \eqref{eq:ramu} where it is clear that
\begin{equation}
\label{eq:tiyo}
E(u)\geq \|{\rm Op}(e^{-\La})Au\|^2+\|\mD^{\kappa} {\rm Op}(e^{-\La}) u\|^2.
\end{equation}
It is easy to check that
\begin{equation}
\label{eq:sugi}
\begin{split}
\frac{d}{dt}\|{\rm Op}(e^{-\La})Au\|^2=-2{\mathsf{ Re}}\,({\rm Op}(\La' e^{-\La})Au,{\rm Op}(e^{-\La})Au)\\
-2 \theta  {\mathsf{ Re}}\,({\rm Op}(e^{-\La})\mD^{\kappa} Au,{\rm Op}(e^{-\La})Au)
\\
-2{\mathsf{ Im}}\,({\rm Op}(e^{-\La})\mD b(t)a\mD u,{\rm Op}(e^{-\La})Au)\\
-2{\mathsf{ Im}}\,({\rm Op}(e^{-\La})\mD b(t){\rm Op}(a_1)\mD u,{\rm Op}(e^{-\La})Au)\\
+2{\mathsf{ Im}}\,({\rm Op}(e^{-\La})Ru,{\rm Op}(e^{-\La})Au)-2{\mathsf{ Im}}\,({\rm Op}(e^{-\La})P^{\sharp}u, {\rm Op}(e^{-\La})Au)
\end{split}
\end{equation}
where $\Lambda'=\dif_t\Lambda$. We have also
\begin{equation}
\label{eq:sugi:b}
\begin{split}
\frac{d}{dt}\|\mD^{\kappa} {\rm Op}(e^{-\La})u\|^2
=-2{\mathsf{ Re}}\,(\mD^{\kappa}{\rm Op}(\La'e^{-\La})u,\mD^{\kappa} {\rm Op}(e^{-\La})u)\\
-2{\mathsf{ Im}}\,(\mD^{\kappa} {\rm Op}(e^{-\La})Au,\mD^{\kappa} {\rm Op}(e^{-\La})u)\\
-2 \theta {\mathsf{ Re}}\,(\mD^{\kappa} {\rm Op}(e^{-\La})\mD^{\kappa} u,\mD^{\kappa} {\rm Op}(e^{-\La})u).
\end{split}
\end{equation}
On the other hand we have
\begin{equation}
\label{eq:sugi:c}
\begin{split}
\frac{d}{dt}{\mathsf{ Re}}\,(b(t)a{\rm Op}(e^{-\La})\mD u,{\rm Op}(e^{-\La})\mD u)\\
={\mathsf{ Re}}\,(b'(t)a(\cdot){\rm Op}(e^{-\La})\mD u,{\rm Op}(e^{-\La})\mD u)\\
-2{\mathsf{ Re}}\,(b(t)a(\cdot){\rm Op}(\La' e^{-\La})\mD u,{\rm Op}(e^{-\La})\mD u)\\
-2{\mathsf{ Im}}\,(b(t)a(\cdot){\rm Op}(e^{-\La})A\mD u,{\rm Op}(e^{-\La})\mD u)\\
-2 \theta {\mathsf{Re}}\,(b(t)a (\cdot){\rm Op}(e^{-\La})\mD^{1+\kappa} u,{\rm Op}(e^{-\La})\mD u).
\end{split}
\end{equation}
%
We denote
\begin{align*}
{\mathcal E}_1&={\mathsf{Re}}\,(b(t)a (\cdot){\rm Op}(e^{-\La})\mD^{1+\kappa} u,{\rm Op}(e^{-\La})\mD u)\\
&\quad \quad+{\mathsf{ Re}}\,(\mD^{\kappa} {\rm Op}(e^{-\La})\mD^{\kappa} u,\mD^{\kappa} {\rm Op}(e^{-\La})u),\\
{\mathcal E}_2&={\mathsf{ Re}}\,({\rm Op}(e^{-\La})\mD^{\kappa} Au,{\rm Op}(e^{-\La})Au)
\end{align*}
and 
\begin{align*}
{\mathcal H}&={\mathsf{Im}}\,\big\{({\rm Op}(e^{-\La})\mD b(t)a\mD u,{\rm Op}(e^{-\La})Au)
\\
&\quad\quad+(b(t)a{\rm Op}(e^{-\La})A\mD u,{\rm Op}(e^{-\La})\mD u)\big\},\\
{\mathcal K}&=-2{\mathsf{ Re}}\,\big\{ (b(t)a{\rm Op}(\La'e^{-\La})\mD u,{\rm Op}(e^{-\Lambda})\mD u)\\
&\quad\quad+(\mD^{\kappa}{\rm Op}(\La'e^{-\La})u,\mD^{\kappa} {\rm Op}(e^{-\La})u)\big\}
\\
&\quad+{\mathsf{ Re}}\,(b'(t) a{\rm Op}(e^{-\Lambda})\mD u,{\rm Op}(e^{-\Lambda})\mD u).
\end{align*}
%
Then one can write
\begin{equation}
\label{eq:samui}
\begin{split}
dE/dt=-2\theta {\mathcal E}_1-2\theta {\mathcal E}_2-2{\mathcal H}+{\mathcal K}\\
-2{\mathsf{ Re}}\,({\rm Op}(\La' e^{-\La})Au,{\rm Op}(e^{-\La})Au)\\
-2{\mathsf{ Im}}\,(\mD^{\kappa} {\rm Op}(e^{-\La})Au,\mD^{\kappa} {\rm Op}(e^{-\La})u)\\
-2{\mathsf{ Im}}\,({\rm Op}(e^{-\La})\mD b(t){\rm Op}(a_1)\mD u,{\rm Op}(e^{-\La})Au)\\
+2{\mathsf{ Im}}\,({\rm Op}(e^{-\La})Ru,{\rm Op}(e^{-\La})Au)-2{\mathsf{ Im}}\,({\rm Op}(e^{-\La})P^{\sharp}u, {\rm Op}(e^{-\La})Au).
\end{split}
\end{equation}
In estimating $dE/dt$, a term which is bounded by 
\begin{equation}
\label{eqivi}
C\mu^{\epsilon}\big(\|\mD^{3\kappa/2}{\rm Op}(e^{-\La}) u\|^2+\|\mD^{\kappa/2}{\rm Op}(e^{-\La}) Au\|^2\big)
\end{equation}
with some $\epsilon>0$ and $C>0$ is irrelevant, choosing $\mu>0$ small (see Lemmas \ref{lem:hirame} and \ref{levki} below)   then we write
\[
A\preceq B 
\]
if the inequality $A\leq B$ holds modulo terms bounded by  \eqref{eqivi}. If $A=B$ holds modulo terms bounded by  \eqref{eqivi} we denote $A\simeq B$.

We start with estimating   ${\mathcal E}_1$.
\begin{lemma}
\label{lem:levi}
We have
\begin{align*}
&{\mathcal E}_1\simeq {\mathsf{ Re}}\big(b(t)\,{\rm Op}(\phi \xim^{\kappa}){\rm Op}(e^{-\La}) \mD u,{\rm Op}(e^{-\La}) \mD u\big)
\\
&
\simeq \|{\rm Op}(\sqrt{b(t)}\sqrt{\phi}\xim^{\kappa/2}){\rm Op}(e^{-\La})\mD u\|^2.
\end{align*}
\end{lemma}
\begin{proof}Note that $\xim^{\kappa}\in S_{\phi}(\xim^{\kappa})\cap S_{\delta}^{\lr{s}}(\xim^{\kappa})$. From Proposition \ref{pro:abLam} taking $N$ large, one can write 
\[
e^{-\La}\# \xim^{\kappa}=(\xim^{\kappa}+a_1+a_2)e^{-\La}+qe^{-\Lambda}+r,\quad r\in  S_{0,0}^{(\bas)}(e^{-c\xim^{\baka}})
\]
where $a_j\in  S_{\phi}(\xim^{\kappa-j+{\tilde \kappa}j}\phi^{-j/2})\cap S_{\delta}^{\lr{s}}(\xim^{\kappa-(\kappa-\tika)j})$ by Proposition \ref{pro:abLam} and Lemma \ref{lem:moriyama:bis} and $q\in  S_{\delta}^{\lr{s}}(\xim^{-M})$ where $\bas \tika<1$, $\baka>\tika$ and $a_1$ is pure imaginary. Thanks to Lemma \ref{lem:koreha} one can write 
\[
(\xim^{\kappa}+a_1+a_2)e^{-\La}=(\xim^{\kappa}+A_1+A_2)\#e^{-\Lambda}+q'e^{-\Lambda}+r'
\]
where $q'$ and $r'$ enjoy the same properties as $q$ and $r$ and hence
\begin{equation}
\label{eq:hiru}
e^{-\La}\# \xim^{\kappa}=(\xim^{\kappa}+A_1+A_2)\#e^{-\Lambda}+{\tilde q}e^{-\Lambda}+{\tilde r}
\end{equation}
 where $A_j\in  S_{\phi}(\xim^{\kappa+{\tilde \kappa}j-j}\phi^{-j/2})$ and that ${\tilde q}\in  S_{\delta}^{\lr{s}}(\xim^{-M})$, ${\tilde r}\in S_{0,0}^{(\bas)}(e^{-c\xim^{\baka}})$ with possibly different  $\bas$, $\baka$ such that $\bas\tika<1$ and $\baka>\tika$. Note that $A_1$ is pure imaginary and then  
\begin{equation}
\label{eq:hiru:b}
a\#(\xim^{\kappa}+A_1+A_2)=a\xim^{\kappa}+{\tilde A}_1+{\tilde A}_2
\end{equation}
where ${\tilde A}_1\in  S_{\phi}(\xim^{\kappa+{\tilde \kappa}-1}\sqrt{\phi})$  is pure imaginary and ${\tilde A}_2\in S_{\delta}(\xim^{\kappa+2{\tilde \kappa}-2})$ because $a\in S_{\phi}(\phi)$. Therefore from \eqref{eq:hiru} and \eqref{eq:hiru:b} one can write
\[
a\#e^{-\La}\# \xim^{\kappa}=Q\#e^{-\Lambda}+a\#({\tilde q}e^{-\Lambda}+{\tilde r}),\quad {\tilde r}\in S_{0,0}^{(\bas)}(e^{-c\xim^{\baka}})
\]
with $\bas\tika<1$ and $\baka>\tika$ where 
\[
{\mathsf{ Re}}\,Q-a(x)\xim^{\kappa}\in  S_{\delta}(\xim^{-2+\kappa+2{\tilde \kappa}})\subset \mu^{2(\kappa-\tika)}S_{\delta}(\xim^{-2+3\kappa})
\]
 since ${\tilde A}_1$ is pure imaginary. Therefore taking Lemma \ref{leiva} and Corollary \ref{cor:koike} into account one has
\begin{equation}
\label{eqve}
\begin{split}
{\mathsf{ Re}}(a{\rm Op}(e^{-\La}) \mD^{1+\kappa}u,{\rm Op}(e^{-\La}) \mD u)\\
\simeq {\mathsf{ Re}}({\rm Op}(a\xim^{\kappa}){\rm Op}(e^{-\La})\mD u,{\rm Op}(e^{-\La})\mD u).
\end{split}
\end{equation}
On the other hand, thanks to Lemmas \ref{lem:sanuki} and \ref{lem:yodobasi} one has 
\begin{align*}
|(a{\rm Op}({\tilde q}e^{-\Lambda}+{\tilde r})\mD u,{\rm Op}(e^{-\La})\mD u)|
\leq C\mu \|{\rm Op}(e^{-\La})u\|^2
\end{align*}
%
%
%
%
%
choosing $M\geq 2$. From \eqref{eq:hiru} one obtains
\[
\xim^{2\kappa}\#e^{-\La}\# \xim^{\kappa}=(\xim^{3\kappa}+{\tilde A}_1+{\tilde A}_2)\#e^{-\Lambda}+\xim^{2\kappa}\#({\tilde q}e^{-\Lambda}+{\tilde r})
\]
where ${\tilde A}_1\in S_{\delta}(\xim^{2\kappa+\tika})$ is pure imaginary and ${\tilde A}_2\in S_{\delta}(\xim^{\kappa+2\tika})\subset \mu^{2(\kappa-\tika)}S_{\delta}(\xim^{3\kappa})$ because $\kappa+\delta=1$ and $\phi^{-j/2}\in S_{\phi}(\xim^{j\delta})$. Therefore 
\begin{align*}
{\mathsf{ Re}}\,(\mD^{\kappa} {\rm Op}(e^{-\La})\mD^{\kappa} u,\mD^{\kappa} {\rm Op}(e^{-\La})u)\simeq
\|\mD^{3\kappa/2}{\rm Op}(e^{-\La})u\|^2.
\end{align*}
Repeating the same arguments proving \eqref{eq:hiru} one can write
\[
e^{-\La}\#\xim^{-1}=(\xim^{-1}+A_1)\#e^{-\La}+{\tilde q}e^{-\Lambda}+{\tilde r}
\]
where $A_1\in S_{\phi}(\xim^{\tika-2}\phi^{-1/2})\subset S_{\delta}(\xim^{\tika-2+\delta})$. Thus we have
\begin{align*}
\|\mD^{3\kappa/2}{\rm Op}(e^{-\La})u\|=\|\mD^{3\kappa/2}{\rm Op}(e^{-\La})\mD^{-1}\mD u\|\\
\simeq \|\mD^{3\kappa/2-1}{\rm Op}(e^{-\La})\mD u\|={\mathsf{ Re}}\,(\mD^{3\kappa-2}{\rm Op}(e^{-\La})\mD u,{\rm Op}(e^{-\La})\mD u)
\end{align*}
because $\xim^{3\kappa/2}\#A_1\in S_{\delta}(\xim^{\kappa/2+\tika-1})\subset \mu^{\kappa-\tika}S_{\delta}(\xim^{3\kappa/2-1})$. Since $3\kappa-2=-2\delta+\kappa$ and  $a(x)\xim^{\kappa}+\xim^{-2\delta+\kappa}=(a(x)+\xim^{-2\delta})\xim^{\kappa}=\phi\xim^{\kappa}$ we have the first assertion. 

To prove the second assertion  we note
\[
(\phi^{1/2}\xim^{\kappa/2})\#(\phi^{1/2}\xim^{\kappa/2})-\phi \xim^{\kappa}\in  S_{\phi}(\xim^{-2+\kappa})\subset \mu^{2\kappa}S_{\delta}(\xim^{-2+3\kappa})
\]
since $\phi^{1/2}\xim^{\kappa/2}\in S_{\phi}(\phi^{1/2}\xim^{\kappa/2})$ then using Corollary \ref{cor:koike} we have 
\begin{align*}
\|\sqrt{b(t)}{\rm Op}(\sqrt{\phi}\xim^{\kappa/2}){\rm Op}(e^{-\La})\mD u\|^2\\
\equiv (b(t){\rm Op}(\phi\xim^{\kappa}){\rm Op}(e^{-\La})\mD u,{\rm Op}(e^{-\La})\mD u)
\end{align*}
which proves the second assertion.
\end{proof}
\begin{lemma}We have
\[
{\mathcal E}_1\geq (1-C\mu^{\kappa-\tika})\|\mD^{3\kappa/2}{\rm Op}(e^{-\La})u\|^2.
\]
\label{lem:hirame}
\end{lemma}
\begin{proof}Since $0\leq a(x)\xim^{\kappa}\in S(\xim^{\kappa},dx^2+\xim^{-2}d\xi^2)$ then from the Fefferman-Phong inequality it follows that
\begin{align*}
{\mathsf{ Re}}({\rm Op}(a\xim^{\kappa}){\rm Op}(e^{-\La})\mD u,
{\rm Op}(e^{-\La})\mD u)\\
\geq -C\|\xim^{\kappa/2-1}{\rm Op}(e^{-\La})\mD u\|^2\geq -C'\|\xim^{\kappa/2}{\rm Op}(e^{-\La}) u\|^2.
\end{align*}
The rest of the proof is clear from Lemma \ref{levki}.
\end{proof}
%
\begin{lemma}
\label{lem:nihobi}
Assume that $T\in S_{\phi}(\sqrt{\phi}\xim^{\kappa/2})$  is real or pure imaginry. Then there is $C>0$ such that
\[
\|b(t) {\rm Op}(T){\rm Op}(e^{-\Lambda})\mD u\|^2 \preceq C{\mathcal E}_1.
\]
\end{lemma}
\begin{proof}We may assume that $T$ is real. Note ${\bar T}\#T=T^2+R$ with $R\in S_{\phi}(\xim^{-2+\kappa})\subset \mu^{2\kappa}S_{\delta}(\xim^{-2+3\kappa})$ 
and by  Corollary \ref{cor:koike} one has 
\[
({\rm Op}(R){\rm Op}(e^{-\Lambda})\mD u,{\rm Op}(e^{-\Lambda})\mD u)\simeq 0.
\]
 Write $
C\phi\xim^{\kappa}-T^2
=Cq^2$ with $q=\xim^{\kappa/2}\sqrt{\phi}\sqrt{1-C^{-1}T^2\xim^{-\kappa}\phi^{-1}}$
where $q\in S_{\phi}(\xim^{\kappa/2}\sqrt{\phi})$ if we take a large $C>0$. Since $\phi\in S_{\phi}(\phi)$ we have
\[
q\#q-q^2=r\in  S_{\phi}(\xim^{\kappa-2})\subset \mu^{2\kappa}S_{\delta}(\xim^{-2+3\kappa})
\]
and $|({\rm Op}(r){\rm Op}(e^{-\Lambda})\mD,{\rm Op}(e^{-\Lambda})\mD)|\simeq 0$ by Corollary \ref{cor:koike}. The rest of the proof is clear.
\end{proof}

\begin{proposition}
\label{prva} 
There exists $C>0$ such that
\begin{align*}
|{\mathcal H}|\preceq C\mu^{\kappa-\tika}{\mathcal E}_1.
\end{align*}
\end{proposition} 
\begin{proof} We first write 
\begin{eqnarray*}
{\mathcal H}={\mathsf{ Im}}\,(b(t){\rm Op}(e^{-\La})A\mD u,[a,{\rm Op}(e^{-\La})]\mD u)\\
+{\mathsf{ Im}}\,(b(t) [{\rm Op}(e^{-\La}),\mD]a\mD u,{\rm Op}(e^{-\La})Au)\\
+{\mathsf{ Im}}\,(b(t){\rm Op}(e^{-\La})a\mD u,[\mD,{\rm Op}(e^{-\La})]Au)\\={\mathsf{ Im}}\,{\mathcal H}_1+{\mathsf{ Im}}\,{\mathcal H}_2+{\mathsf{ Im}}\,{\mathcal H}_3
\end{eqnarray*}
because $\mD A=A \mD$. Then to prove the proposition it suffices to prove there is $C>0$ such that 
\begin{equation}
\label{eq:izumo}
|{\mathsf{ Im}}\,{\mathcal H}_j|\preceq C\mu^{\kappa-\tika}{\mathcal E}_1
\end{equation}
with some $C>0$ for $j=1, 2, 3$.  We first study 
\[
{\mathcal H}_1=(b(t){\rm Op}(e^{-\Lambda})A\mD u,[a,{\rm Op}(e^{-\Lambda})]\mD u).
\]
 From Proposition \ref{pro:abLam} and Lemma \ref{lem:koreha} one can write
\begin{equation}
\label{eq:aLamu}
e^{-\Lambda}\#a-a\#e^{-\Lambda}=(B_1+B_2)\#e^{-\Lambda}+qe^{-\Lambda}+r
\end{equation}
where $q\in S_{\delta}^{\lr{s}}(\xim^{-M})$  and $r\in S_{0,0}^{\lr{\bas}}(e^{-c\xim^{\baka}})$ with $\bas\tika<1$ and $\baka>\tika$. Note that $B_j\in S_{\phi}(\phi^{1-j/2}\xim^{{\tilde \kappa}j-j})\cap S_{\delta}^{\lr{s}}(1)$ since $a\in S_{\phi}(\phi)\cap S_{\delta}^{\lr{s}}(1)$. Note that
\begin{eqnarray*}
&|({\rm Op}(e^{-\Lambda})A\mD u, {\rm Op}(B_2){\rm Op}(e^{-\Lambda})\mD u)|\\
&\leq \|\mD^{-1+\kappa/2} {\rm Op}(e^{-\Lambda})\mD A u\|\|\mD^{1-\kappa/2} {\rm Op}(B_2){\rm Op}(e^{-\Lambda})\mD u\|\simeq 0
\end{eqnarray*}
by Corollary \ref{cor:koike} because $B_2\in S_{\phi}(\xim^{2\tika-2})\subset \mu^{2(\kappa-\tika)}S_{\delta}(\xim^{2\kappa-2})$. 
Noting  that $\xim^{1-\kappa/2} \#B_1\in  S_{\phi}(\sqrt{\phi}\xim^{{\tilde \kappa}-\kappa/2})\subset \mu^{\kappa-\tika} S_{\phi}(\sqrt{\phi}\xim^{\kappa/2})$ from Lemma \ref{lem:nihobi} one has  
\begin{align*}
\|b(t)\mD {\rm Op}( B_1){\rm Op}(e^{-\Lambda})\mD u\|^2 \preceq C\mu^{2(\kappa-\tika)}{\mathcal E}_1
\end{align*}
with some $C>0$. Therefore \eqref{eq:izumo} for $j=1$ is proved. 
We turn to study
\begin{align*}
&{\mathcal H}_2=(b(t) [{\rm Op}(e^{-\Lambda}),\mD]a\mD u,{\rm Op}(e^{-\Lambda})Au)\\
&=(b(t) \mD^{-\kappa/2}[{\rm Op}(e^{-\Lambda}),\mD]a\mD u,\mD^{\kappa/2}{\rm Op}(e^{-\Lambda})Au)
\end{align*}
where the right-hand side is estimated by 
\[
C\|\mD^{\kappa/2-\epsilon}{\rm Op}(e^{-\Lambda})Au\|^2+C\|b(t) \mD^{-\kappa/2+\epsilon}[{\rm Op}(e^{-\Lambda}),\mD]a\mD u\|^2
\]
where $\epsilon=(\kappa-\tika)/2$. 
We now estimate $\|b(t)\mD^{-\kappa/2}[{\rm Op}(e^{-\Lambda}),\mD]a\mD u\|$.  Taking $N$ large so that $(\kappa-\tika)N>M+1$ in Proposition \ref{pro:abLam} one can write thanks to Lemma \ref{lem:koreha} 
\begin{equation}
\label{eq:dyson}
e^{-\Lambda}\#\xim-\xim\#e^{-\Lambda}=(A_1+A_2)\#e^{-\Lambda}+qe^{-\Lambda}+R\;\; 
\end{equation}
where $A_j\in  S_{\phi}(\xim^{{\tilde \kappa}j+1-j}\phi^{-j/2})\cap S_{\delta}^{\lr{s}}(\xim^{1-(\kappa-\tika)j})$, $q\in  S_{\delta}^{\lr{s}}(\xim^{-M})$ and $R\in  S_{0,0}^{\lr{\bas}}(e^{-c\xim^{\baka}})$ with $\bas\tika<1$, $\baka>\tika$.   It follows from Lemmas \ref{lem:sanuki} and \ref{lem:yodobasi} that  $\|\mD^{-\kappa/2+\epsilon}{\rm Op}(qe^{-\La}+R)a\mD u\|\simeq 0$.  From \eqref{eq:aLamu} one can write $
e^{-\Lambda}\#a=\big(a+B_1+B_2)\#e^{-\Lambda}+qe^{-\Lambda}+r$ then
\[
(A_1+A_2)\#e^{-\Lambda}\#a=({\tilde A}_1+{\tilde A}_2)\#e^{-\Lambda}+(A_1+A_2)\#(qe^{-\Lambda}+r)
\]
where  ${\tilde A}_j\in S_{\phi}(\phi^{1-j/2}\xim^{{\tilde \kappa}j+1-j})$. It follows from Lemma \ref{lem:yodobasi} and Lemma \ref{lem:sanuki} that
\[
\|\mD^{-\kappa/2+\epsilon}{\rm Op}(A_1+A_2){\rm Op}(qe^{-\Lambda}+r)a\mD u\|\simeq 0.
\]
Since $\xim^{-\kappa/2+\epsilon}\#{\tilde A}_2\in S_{\delta}(\xim^{3{\tilde \kappa}/2})\subset \mu^{3\epsilon}S_{\delta}(\xim^{3\kappa/2-1})$  then from Corollary \ref{cor:koike} we have
\[
\|b(t)\mD^{-\kappa/2}{\rm Op}({\tilde A}_2){\rm Op}(e^{-\Lambda})\mD u\|\simeq 0.
\]
It remains to estimate $\|b(t)\mD^{-\kappa/2+\epsilon}{\rm Op}({\tilde A}_1){\rm Op}(e^{-\Lambda})\mD u\|$. Since $\xim^{-\kappa/2+\epsilon}\#{\tilde A}_1\in  S_{\phi}(\phi^{1/2}\xim^{-\tika/2})\subset \mu^{\epsilon}S_{\phi}(\phi^{1/2}\xim^{\kappa/2})$ we conclude \eqref{eq:izumo} for $j=2$ from Lemma \ref{lem:nihobi}. 
%
%
We finally consider ${\mathsf{Im}}\,{\mathcal H}_3$. From \eqref{eq:dyson} one can write
\[
[{\rm Op}(e^{-\La}),\mD]={\rm Op}(A_1+A_2){\rm Op}(e^{-\La})+{\rm Op}(qe^{-\La}+R).
\]
Thanks to Lemmas \ref{lem:sanuki} and \ref{lem:yodobasi} one has
\begin{align*}
|({\rm Op}(e^{-\Lambda})a\mD u,{\rm Op}(qe^{-\Lambda}+R)Au)|\\
\leq \|\mD^{-1}{\rm Op}(e^{-\Lambda})a\mD u\|\|\mD {\rm Op}(qe^{-\Lambda}+R) Au\|\simeq 0.
\end{align*}
 Write
\begin{align*}
(b(t){\rm Op}(e^{-\Lambda})a\mD u,{\rm Op}(A_1+ A_2){\rm Op}(e^{-\Lambda})Au)\\
=(b(t)[{\rm Op}(e^{-\Lambda}),a]\mD u, {\rm Op}(A_1+ A_2){\rm Op}(e^{-\Lambda})Au)\\
+(b(t)a {\rm Op}(e^{-\Lambda})\mD u, {\rm Op}(A_1+ A_2){\rm Op}(e^{-\Lambda})Au).
\end{align*}
It is clear from Corollary \ref{cor:koike} that 
\begin{align*}
|b(t)a{\rm Op}(e^{-\Lambda})\mD u, {\rm Op}(A_2){\rm Op}(e^{-\Lambda})Au)|\\
\leq \|b(t)\mD^{-\kappa/2}{\rm Op}({\bar A}_2)a{\rm Op}(e^{-\Lambda})\mD u\|\|\mD^{\kappa/2}{\rm Op}(e^{-\Lambda})Au\|\simeq 0
\end{align*}
because $\xim^{-\kappa/2}\#{\bar A}_2\#a\in  S_{\delta}(\xim^{{\tilde \kappa}-\kappa/2-1})\subset \mu^{\kappa-\tika}S_{\delta}(\xim^{3\kappa/2-1})$. Note that
\begin{align*}
2(b(t)a {\rm Op}(e^{-\Lambda})\mD u, {\rm Op}(A_1){\rm Op}(e^{-\Lambda})Au)\\
\leq \|b(t)\mD^{-\kappa/2+\epsilon}{\rm Op}({\bar A}_1)a{\rm Op}(e^{-\Lambda})\mD u\|^2
+\|\mD^{-\kappa/2-\epsilon}{\rm Op}(e^{-\Lambda})Au\|^2.
\end{align*}
Since $\xim^{-\kappa/2+\epsilon}{\bar A}_1\#a\in  S_{\phi}(\sqrt{\phi}\xim^{\tika- \kappa/2+\epsilon})\subset \mu^{\epsilon}S_{\phi}(\sqrt{\phi}\xim^{\kappa/2})$  one can conclude from Lemma \ref{lem:nihobi}
\begin{eqnarray*}
(b(t)a {\rm Op}(e^{-\Lambda})\mD u, {\rm Op}(A_1+A_2){\rm Op}(e^{-\Lambda})Au)\preceq C\mu^{\kappa-\tika}{\mathcal E}
\end{eqnarray*}
with some $C>0$. Using \eqref{eq:aLamu} we see that 
\[
\big|([{\rm Op}(e^{-\Lambda}),a]\mD u, {\rm Op}(A_1+A_2){\rm Op}(e^{-\Lambda})Au)\big|\simeq 0
\]
in virtue of Lemmas \ref{lem:sanuki} and \ref{lem:yodobasi} and Corollary \ref{cor:koike}. Thus we get \eqref{eq:izumo} for $j=3$.
\end{proof}
%
%

\medskip


%
\begin{proposition}
\label{prvi} We have
\begin{eqnarray*}
{\mathcal K}
\preceq C \mu^{3\kappa/2-\tika}{\mathcal E}_1
\end{eqnarray*}
with some $C>0$.
\end{proposition}
%
%
%
%
%
Since $\Lambda'=\dif_t\Lambda \in S_{\phi}(\sqrt{\phi}\xim^{1-\kappa})\cap S_{\delta}^{\lr{s}}(\xim^{1-\kappa})$ by Lemma \ref{lem:budo:a} applying Lemma \ref{lem:koreha} we get
\begin{equation}
\label{eq:ramA}
\Lambda'e^{-\Lambda}=(\Lambda'+\lambda_1+\lambda_2)\#e^{-\Lambda}+qe^{-\Lambda}+r
\end{equation}
with $\lambda_j\in  S_{\phi}(\phi^{-(j-1)/2}\xim^{-(j-1)-\kappa+{\tilde \kappa}j})\cap S_{\delta}^{\lr{s}}(\xim^{1-\kappa-(\kappa-\tika)j})$,  $j=1,2$  and $q\in  S_{\delta}^{\lr{s}}(\xim^{-M})$, $r\in S_{0,0}^{\lr{\bas}}(e^{-c\xim^{\baka}})$ with $\bas\tika<1$, $\baka>\tika$. Note that $\lambda_1$ is pure imaginary. Since $a\in S_{\phi}(\phi)$ one can write
\[
a\#(\Lambda'+\lambda_1+\lambda_2)=a\Lambda'+Q_1+Q_2
\]
where $Q_j\in S_{\phi}(\phi^{(3-j)/2}\xim^{-(j-1)-\kappa+{\tilde \kappa}j})$ and $Q_1$ is pure imaginary. Note that  $(a\,{\rm Op}(qe^{-\Lambda}+r)\mD u, {\rm Op}(e^{-\Lambda}) \mD u)\simeq 0$ by Lemmas \ref{lem:yodobasi} and \ref{lem:sanuki} therefore  one has
\begin{equation}
\label{eq:isoji}
\begin{split}
{\mathsf{Re}}\,(b(t)a{\rm Op}&(\Lambda'e^{-\Lambda})\mD u, {\rm Op}(e^{-\Lambda})\mD u)\\
\simeq ({\rm Op}&(b(t) a\Lambda'){\rm Op}(e^{-\Lambda})\mD u, {\rm Op}(e^{-\Lambda})\mD u)\\
&+{\mathsf{Re}}\,({\rm Op}(b(t) Q_2){\rm Op}(e^{-\Lambda})\mD u, {\rm Op}(e^{-\Lambda})\mD u).
\end{split}
\end{equation}
Using \eqref{eq:ramA} one has $(\mD^{\kappa}{\rm Op}(\Lambda'e^{-\Lambda})u,\mD^{\kappa}{\rm Op}(e^{-\Lambda})u)
\simeq (\mD^{2\kappa-1}{\rm Op}(\Lambda'+\lambda_1+\lambda_2){\rm Op}(e^{-\Lambda})u,\mD {\rm Op}(e^{-\Lambda})u)$.  By \eqref{eq:dyson} we see
\begin{align*}
(\mD^{2\kappa-1}{\rm Op}(\Lambda'+\lambda_1+\lambda_2){\rm Op}(e^{-\Lambda})u,[\mD, {\rm Op}(e^{-\Lambda})]u)\\
\simeq (\mD^{2\kappa-1}{\rm Op}(\Lambda'+\lambda_1+\lambda_2){\rm Op}(e^{-\Lambda})u, {\rm Op}(A_1+ A_2){\rm Op}(e^{-\Lambda})u)\simeq 0
\end{align*}
since $({\bar A}_1+{\bar A}_2)\#\xim^{2\kappa-1}\#(\Lambda'+\lambda_1+\lambda_2)\in  S_{\phi}(\xim^{\kappa+{\tilde \kappa}})\subset \mu^{2\kappa-\tika}S_{\delta}(\xim^{3\kappa})$ because  $\La'\in S_{\phi}(\sqrt{\phi}\xim^{\delta})$ by Lemma \ref{lem:budo:a} and $\lambda_j\in  S_{\phi}(\phi^{-(j-1)/2}\xim^{-\kappa+{\tilde \kappa}j})$  and $A_j\in  S_{\phi}(\xim^{{\tilde \kappa}j+1-j}\phi^{-j/2})$. Thus one has
\begin{equation}
\label{eq:momo}
\begin{split}
&(\mD^{\kappa}{\rm Op}(\Lambda'e^{-\Lambda})u,\mD^{\kappa}{\rm Op}(e^{-\Lambda})u)\\
&\simeq (\mD^{2\kappa-1}{\rm Op}(\Lambda'+\lambda_1+\lambda_2)e^{-\Lambda}u,{\rm Op}(e^{-\Lambda})\mD u).
\end{split}
\end{equation}
Repeating the same argument one has 
\[
(\mD^{2\kappa-1}{\rm Op}(\Lambda'+\lambda_1+\lambda_2)\mD^{-1}[\mD, {\rm Op}(e^{-\Lambda})]u,{\rm Op}(e^{-\Lambda})\mD u)\simeq 0
\]
and hence 
\begin{eqnarray*}
&(\mD^{\kappa}{\rm Op}(\Lambda'e^{-\Lambda})u,\mD^{\kappa}{\rm Op}(e^{-\Lambda})u)\\
&\simeq (\mD^{2\kappa-1}{\rm Op}(\Lambda'+\lambda_1+\lambda_2)\mD^{-1}{\rm Op}(e^{-\Lambda})\mD u,{\rm Op}(e^{-\Lambda})\mD u).
\end{eqnarray*}
Since $\xim^{2\kappa-1}\#(\Lambda'+\lambda_1+\lambda_2)\#\xim^{-1}-\xim^{2\kappa-2}\Lambda'\in  S_{\delta}(\xim^{-2+\kappa+{\tilde \kappa}})$ which is contained in $\mu^{2\kappa-\tika}S_{\delta}(\xim^{-2+3\kappa})$ we obtain 
\begin{equation}
\label{eq:kaki:b}
\begin{split}
(\mD^{\kappa}{\rm Op}(\Lambda' e^{-\Lambda}) u,\mD^{\kappa}{\rm Op}(e^{-\Lambda})u)\\
\simeq ({\rm Op}(\Lambda'\xim^{2\kappa-2}){\rm Op}(e^{-\Lambda})\mD u, {\rm Op}(e^{-\Lambda})\mD u).
\end{split}
\end{equation}
We now study the sum:
\begin{align*}
({\rm Op}(b(t) a\Lambda'){\rm Op}(e^{-\Lambda})\mD u, {\rm Op}(e^{-\Lambda})\mD u)\\
+({\rm Op}(\Lambda'\xim^{2\kappa-2}){\rm Op}(e^{-\Lambda})\mD u, {\rm Op}(e^{-\Lambda})\mD u).
\end{align*}
Noting
\begin{align*}
b(t)a(x)\Lambda'+\xim^{2\kappa-2}\Lambda'
=\frac{|b'(t)|\phi\,(b(t)a(x)+\xim^{-2\delta})}{b(t)\phi+\xim^{-2\delta}}
\end{align*}
we write
\[
\Delta=2\big(b(t)a(x)\Lambda'+\xim^{2\kappa-2}\Lambda'\big)-b'(t)a(x)\\
=\frac{ F}{b(t)\phi(x,\xi)+\xim^{-2\delta}}
\]
where
\begin{align*}
F=2|b'(t)|\phi\,\big(b(t)a(x)+\xim^{-2\delta}\big)-b'(t)a(x)b(t)\phi
-\xim^{-2\delta}b'(t)a(x)\\
=|b'(t)|\phi\,\big(b(t)a(x)+\xim^{-2\delta}\big)\\
+\Big\{|b'(t)|b(t)a(x)\phi+|b'(t)|\phi \,\xim^{-2\delta}-b'(t)b(t)a(x)\phi
-b'(t)a(x)\xim^{-2\delta}\Big\}\\
\geq |b'(t)|\phi\,\big(b(t)a(x)+\xim^{-2\delta}\big)
\end{align*}
since $\phi\geq a(x)$. Therefore there exists $c>0$ such that
\begin{align*}
\frac{F}{b(t)\phi+\xim^{-2\delta}}\geq \frac{|b'(t)|\phi\,\big(b(t)a(x)+\xim^{-2\delta}\big)}{b(t)\phi+\xim^{-2\delta}}\\
=\frac{|b'(t)|\phi\,\big(b(t)a(x)+\xim^{-2\delta}\big)}{b(t)a(x)+(b(t)+1)\xim^{-2\delta}}
\geq c\,|b'(t)|\phi
\end{align*}
because $0\leq b(t)\leq C$. On the other hand it is clear that $F\leq 3\,|b'(t)|\phi\,(b(t)\phi+\xim^{-2\delta})$ for $a(x)\leq \phi$ and hence
\[
c\,|b'(t)|\phi\leq \frac{F}{b(t)\phi+\xim^{-2\delta}}\leq 3\,|b'(t)|\phi.
\]
\begin{lemma}
\label{lem:kihon}
 We have
 \[
\Big( \frac{|b'(t)|^{-1}\phi^{-1}F}{b(t)\phi+\xim^{-2\delta}}\Big)^{1/2}\in S_{\phi}(1)
\]
uniformly in $t$ with $|b'(t)|\neq 0$. That is, for any $k, l$ there exists $C_{kl}$ such that
\[
\Big|\dif_x^k\dif_{\xi}^l\Big( \frac{|b'(t)|^{-1}\phi^{-1}F}{b(t)\phi+\xim^{-2\delta}}\Big)^{1/2}\Big|\leq C_{kl}\phi^{-k/2}\xim^{-l}
\]
for any $t\in [0,T]$ with $b'(t)\neq 0$.
\end{lemma}
\begin{proof}
Assume $b'(t)\neq 0$. Set $A=|b'(t)|^{-1}$ and $B=b(t)$. We first note that
\[
|\dif_x^k\dif_{\xi}^l(B\phi+\xim^{-2\delta})|\leq C_{kl}(B\phi+\xim^{-2\delta})\phi^{-k/2}\xim^{-l}
\]
because $\phi, \xim^{-2\delta}\in S(\phi,g)$ from which it follows that
\[
|\dif_x^k\dif_{\xi}^l(B\phi+\xim^{-2\delta})^{-1}|\leq C_{kl}(B\phi+\xim^{-2\delta})^{-1}\phi^{-k/2}\xim^{-l}.
\]
Since 
\[
A\phi^{-1}F=2(Ba(x)+\xim^{-2\delta})-\frac{b'(t)}{|b'(t)|}\Big(Ba(x)+\phi^{-1}a\xim^{-2\delta}\Big)
\]
and $a(x)\in S(\phi,g)$ it is easy to see
\[
|\dif_x^k\dif_{\xi}^lA\phi^{-1}F|\leq C_{kl}(B\phi+\xim^{-2\delta})\phi^{-k/2}\xim^{-l}.
\]
Therefore we conclude that
\begin{equation}
\label{eq:sezon}
|\dif_x^k\dif_{\xi}^l\big(A\phi^{-1}F/(B\phi+\xim^{-2\delta})\big)|\leq C_{kl}\phi^{-k/2}\xim^{-l}
\end{equation}
where $C_{kl}$ is independent of $A$ and $B$. Since $c\leq A\phi^{-1}F/(B\phi+\xim^{-2\delta})\leq 3$ the proof follows from \eqref{eq:sezon}.
\end{proof}
Define $\Gamma$ by
\[
\Gamma=\left\{\begin{array}{ll}
\displaystyle{ |b'(t)|^{1/2}\phi^{1/2}\Big(\frac{|b'(t)|^{-1}\phi^{-1}F}{b(t)\phi+\xim^{-2\delta}}\Big)^{1/2}}\;\;& \text{if}\;\;b'(t)\neq 0\\
 0\;\;&\text{if}\;\;b'(t)=0
 \end{array}\right.
\]
then thanks to Lemma \ref{lem:kihon} we have $
\Gamma\in S_{\phi}( \sqrt{\phi})$ 
uniformly in $t$. Since $\Delta=\Gamma^2$ we can write
\[
\Delta=\Gamma\#\Gamma+R
\]
where $R\in S_{\phi}(\xim^{-2})$ uniformly in $t$ from which one has 
\[
({\rm Op}(\Delta){\rm Op}(e^{-\Lambda})\mD u,{\rm Op}(e^{-\Lambda})\mD u)\geq -C\|\mD^{-1}{\rm Op}(e^{-\Lambda})\mD u\|^2
\]
which proves
\begin{equation}
\label{eq:isoji:b}
\begin{split}
-2\big({\rm Op}(b(t) a\Lambda'+\xim^{2\kappa-2}\Lambda'){\rm Op}(e^{-\Lambda})\mD u, {\rm Op}(e^{-\Lambda})\mD u\big)\\
+{\mathsf{ Re}}\,(b'(t) a{\rm Op}(e^{-\Lambda})\mD u,{\rm Op}(e^{-\Lambda})\mD u)\preceq 0.
\end{split}
\end{equation}

We get back to estimate the second term on the right-hand side of \eqref{eq:isoji} which is estimated as
\begin{eqnarray*}
&|({\rm Op}(b(t) Q_2){\rm Op}(e^{-\Lambda})\mD u, {\rm Op}(e^{-\Lambda})\mD u)|\\
&\leq C\big( \|b(t)\mD^{1-\tika} {\rm Op}( Q_2){\rm Op}(e^{-\Lambda})\mD u\|^2+\|\mD^{-1+\tika}{\rm Op}(e^{-\Lambda})\mD u\|^2\big).
\end{eqnarray*}
Since $\xim^{1-\tika}\#Q_2\in  S_{\phi}(\sqrt{\phi}\xim^{-\kappa+\tika})\subset \mu^{3\kappa/2-\tika}S_{\phi}(\sqrt{\phi}\xim^{\kappa/2})$ then from Lemma \ref{lem:nihobi} it follows that
\begin{equation}
\label{eq:isoji:c}
{\mathsf{Re}}\,({\rm Op}(b(t) Q_2)e^{-\Lambda}\mD u, e^{-\Lambda}\mD u)\preceq C\mu^{3\kappa/2-\tika}{\mathcal E}_1
\end{equation}
 with some $C>0$.
We end the proof of Proposition \ref{prvi} combining \eqref{eq:isoji}, \eqref{eq:kaki:b}, \eqref{eq:isoji:b} and \eqref{eq:isoji:c}.

\medskip

\begin{lemma}
\label{lem:muneage}There exists $C>0$ such that 
\[
({\rm Op}(\Lambda'){\rm Op}(e^{-\Lambda})u,{\rm Op}(e^{-\Lambda})u)\geq -C\mu^{\kappa}\| {\rm Op}(e^{-\Lambda})u\|^2.
\]
\end{lemma}
\begin{proof}Recall $\Lambda'\in S_{\phi}(\phi^{1/2}\xim^{\delta})$ by Lemma \ref{lem:budo:a}. 
Since $h=\big(\sup {g/g^{\sigma}}\big)^{1/2}= \phi^{-1/2}\xim^{-1}\leq  \mu^{\kappa}\phi^{-1/2}\xim^{-\delta}$ by \eqref{eq:hikoshi} then we see that  $\mu^{-\kappa}\Lambda'\in S_{\phi}(1/h)$ hence the sharp G\aa rding inequality (see \cite[Theorem 18.6.2]{Ho}) proves the assertion.
\end{proof}
\begin{lemma}
\label{levki}
There exists $C>0$ such that
\begin{eqnarray}
\label{eqvtt}
&{\mathsf{ Re}}({\rm Op}(\La'e^{-\La})Au,{\rm Op}(e^{-\Lambda})Au)\geq 
-C \mu^{\kappa-\tika}\|{\rm Op}(e^{-\Lambda})Au\|^2,\\
 \label{eqvte}
 &{\mathcal E}_2\geq 
 (1-C\mu^{\kappa-\tika})\,\|\mD^{\kappa/2}{\rm Op}(e^{-\Lambda})Au\|^2,\\
\label{eqvna}
& {\mathsf{ Re}}(\mD^{\kappa}{\rm Op}(e^{-\Lambda})\mD^{\kappa}u,\mD^{\kappa}{\rm Op}(e^{-\Lambda})u)\\\nonumber
&\qquad \quad\geq 
 (1-C\mu^{\kappa-\tika})\,\|\mD^{3\kappa/2}{\rm Op}(e^{-\Lambda})u\|^2. 
\end{eqnarray}
\end{lemma}
\begin{proof}
From \eqref{eq:ramA} one has
\[
\Lambda' e^{-\Lambda}=(\Lambda'+\lambda)\# e^{-\Lambda}+qe^{-\Lambda}+r,\quad \lambda\in  S_{\phi}(\xim^{-\kappa+{\tilde \kappa}})
\]
where $q\in  S_{\delta}^{\lr{s}}(\xim^{-M})$ and $r\in S_{0,0}^{\lr{\bas}}(e^{-c\xim^{\baka}})$ with $\bas\tika<1$, $\baka>\tika$. Noting $\lambda\in \mu^{\kappa-\tika}S_{\delta}(1)$ then \eqref{eqvtt} follows from Lemma \ref{lem:muneage} and Lemmas \ref{lem:sanuki} and \ref{lem:yodobasi}. 

We next consider ${\mathcal E}_2={\mathsf{ Re}}\,({\rm Op}(e^{-\Lambda})\mD^{\kappa}Au,{\rm Op}(e^{-\Lambda})Au)$. From the same arguments proving \eqref{eq:dyson} one can write
\[
e^{-\La}\#\xim^{\kappa}=(\xim^{\kappa}+A_1)\#e^{-\La}+qe^{-\Lambda}+R
\]
where $A_1\in  S_{\phi}(\xim^{\kappa-1+{\tilde \kappa}}\phi^{-1/2})\subset S_{\delta}(\xim^{{\tilde \kappa}})\subset \mu^{\kappa-\tika}S_{\delta}(\xim^{\kappa})$ and $q\in S_{\delta}^{\lr{s}}(\xim^{-M})$ and $R\in S_{0,0}^{\lr{\bas}}(e^{-c\xim^{\baka}})$ with $\bas\tika<1$, $\baka>\tika$. 
Noting
\begin{align*}
{\mathsf{Re}}\,({\rm Op}(\xim^{\kappa}+ A_1){\rm Op}(e^{-\Lambda})A u, {\rm Op}(e^{-\Lambda}) Au)\geq (\mD^{\kappa}{\rm Op}(e^{-\Lambda})A u,{\rm Op}(e^{-\Lambda})A u)\\
-\big|{\mathsf{Re}}\,({\rm Op}(A_1){\rm Op}(e^{-\Lambda})A u,{\rm Op}(e^{-\Lambda})A u)\big|\geq (1-C\mu^{\kappa-\tika} )\|\mD^{\kappa/2}{\rm Op}(e^{-\Lambda})A u\|^2
\end{align*}
and applying Lemmas \ref{lem:yodobasi} and \ref{lem:sanuki} to ${\rm Op}(qe^{-\La}+r)$ we conclude \eqref{eqvte}. The last assertion is proved similarly. 
 \end{proof}
 %

%
\begin{proposition}
\label{prvo}
We have
\begin{eqnarray*}
\big|2{\mathsf{Im}}\,(b(t){\rm Op}(e^{-\Lambda}) \mD {\rm Op}(a_1)\mD u,{\rm Op}(e^{-\Lambda}) Au)\big|
\preceq  \tau\sqrt{c^*}({\mathcal E}_1+{\mathcal E}_2)
\end{eqnarray*}
where 
\begin{equation}
\label{eq:btoa}
c^*=2\,\sup_{t\in [0,T]}b(t)\,\sup_{x\in\R}\dif_x^2a(x).
\end{equation}
\end{proposition}
\begin{proof}  
Recalling \eqref{eq:dyson} and taking  Lemmas \ref{lem:sanuki} and \ref{lem:yodobasi} into account we see
\begin{align*}
\|\mD^{-\kappa/2}{\rm Op}(e^{-\Lambda})\mD {\rm Op}(a_{1})\mD u\|^2\\
\simeq \|\mD^{-\kappa/2}{\rm Op}(\xim+A_1+A_2){\rm Op}(e^{-\Lambda}) {\rm Op}(a_{1})\mD u\|^2.
\end{align*}
because $a_1= (\tau-\theta t) D_xa\dif_{\xi}\xim^{\kappa}\in   S_{\phi}(\phi^{1/2}\xim^{\kappa-1})\cap S_{\delta}^{\lr{s}}(\xim^{\kappa-1})$. From Proposition \ref{pro:abLam} and Lemma \ref{lem:koreha} we have
\begin{equation}
\label{eq:aLamu:b}
a_{1}\#e^{-\Lambda}-e^{-\Lambda}\#a_{1}=(b_1+b_2)\#e^{-\Lambda}+qe^{-\Lambda}+R
\end{equation}
where $b_j\in S_{\phi}(\phi^{(1-j)/2}\xim^{{\tilde \kappa}j+\kappa-1-j})\cap S_{\delta}^{\lr{s}}(\xim^{-(\kappa-\tika)j})$ and  $q\in S_{\delta}^{\lr{s}}(\xim^{-M})$, $R\in S_{0,0}^{\lr{\bas}}(e^{-c\xim^{\baka}})$ with $\bas\tika<1$, $\baka>\tika$.   Repeating the same arguments as before we conclude that
\begin{align*}
\|\mD^{-\kappa/2}{\rm Op}(e^{-\Lambda})\mD {\rm Op}(a_{1})\mD u\|^2\\
\simeq \|\mD^{-\kappa/2}{\rm Op}(\xim+A_1+A_2){\rm Op}(a_{1}+b_1+b_2){\rm Op}(e^{-\Lambda}) \mD u\|^2.
\end{align*}
Since $b_j\in S_{\phi}(\phi^{(1-j)/2}\xim^{{\tilde \kappa}j+\kappa-1-j})$ and $A_j\in S_{\phi}(\phi^{-j/2}\xim^{{\tilde \kappa}j+1-j})$ one has
\begin{align*}
\xim^{-\kappa/2}\#(\xim+A_1+A_2)\#(a_{1}+b_1+b_2)
=\xim^{1-\kappa/2}a_{1}+B
\end{align*}
with $B\in S_{\delta}(\xim^{\kappa/2+{\tilde \kappa}-1})\subset \mu^{\kappa-\tika}S_{\delta}(\xim^{-1+3\kappa/2}$ and $\|{\rm Op}(B){\rm Op}(e^{-\Lambda})\mD u\|^2\simeq 0$ in virtue of Corollary \ref{cor:koike}.  Noticing $\xim^{1-\kappa/2}a_{1}\in S_{\phi}(\sqrt{\phi}\xim^{\kappa/2})$  and hence $(\xim^{1-\kappa/2}a_{1})\#(\xim^{1-\kappa/2}a_{1})-\xim^{2-\kappa}a_1^2\in S_{\phi}(\xim^{\kappa-2})$ we have 
\begin{align*}
\|b(t){\rm Op}(\xim^{1-\kappa/2}a_{1}){\rm Op}(e^{-\Lambda})\mD u\|^2\\
\simeq ({\rm Op}(b(t)^2\xim^{2-\kappa}a_1^2){\rm Op}(e^{-\Lambda})\mD u,{\rm Op}(e^{-\Lambda})\mD u).
\end{align*}
Note that $
 \tau^2c^*\xim^{\kappa}a(x)\geq b(t)\xim^{2-\kappa}a_1^2=b(t)\xim^{2-\kappa}((\tau-\theta t)\dif_xa(x)\dif_{\xi}\xim^{\kappa})^2$ 
because $|\dif_{\xi}\xim^{\kappa}|\leq \xim^{\kappa-1}$  
and $c^*a(x)\geq b(t)(\dif_xa(x))^2$ for $(t,x)\in [0,T]\times \R$ by Glaeser's inequality. Since $0\leq \tau^2c^*\xim^{\kappa}\phi-b(t)\xim^{2-\kappa}a_1^2\in S(\xim^{\kappa},dx^2+\xim^{-2}d\xi^2)$ then from the Fefferman-Phong inequality it follows that
\begin{align*}
\tau^2c^*{\mathsf{ Re}}\big(b(t)\,{\rm Op}(\phi \xim^{\kappa})w,w\big)
 -({\rm Op}(b(t)^2\xim^{2-\kappa}a_1^2)w,w)\\
 \geq -C\|\xim^{\kappa/2-1}w\|^2\simeq 0
\end{align*}
where $w={\rm Op}(e^{-\Lambda})\mD u$. Thanks to Lemmas \ref{lem:levi} and \ref{levki} one has 
\begin{align*}
\big|2{\mathsf{Im}}\,(b(t){\rm Op}(e^{-\Lambda}) \mD a_1\mD u,{\rm Op}(e^{-\Lambda}) Au)\big|\\
\leq (\tau\sqrt{c^*})^{-1}\|b(t)\mD^{-\kappa/2}{\rm Op}(e^{-\Lambda})\mD {\rm Op}(a_{1})\mD u\|^2\\
+\tau\sqrt{c^*}\|\mD^{\kappa/2}{\rm Op}(e^{-\Lambda}) Au\|^2\preceq \tau\sqrt{c^*}({\mathcal E}_1+{\mathcal E}_2)
\end{align*}
which proves the assertion.
\end{proof}
We study $R$ in \eqref{eq:dedasi}. Recall $R=b(t)\mD{\rm Op}(q+r)\mD$.  One has
\begin{align*}
|({\rm Op}(e^{-\Lambda})\mD{\rm Op}(q)\mD u,{\rm Op}(e^{-\Lambda})Au)|\\
\leq C\|\mD^{1-\kappa/2}{\rm Op}(e^{-\Lambda}){\rm Op}(q)\mD u\|
\|\mD^{\kappa/2}{\rm Op}(e^{-\Lambda})Au\|.
\end{align*}
From Corollary \ref{cor:nao} one can write $e^{-\Lambda}\#q=c\,e^{-\Lambda}+{\tilde r}$ with $c\in S_{\delta}^{\lr{s}}(\xim^{-2+2\kappa})$ and ${\tilde r}\in S_{0,0}^{\lr{\bas}}(e^{-c\xim^{\baka}})$ with $\bas\tika<1$, $\baka>\tika$. Recalling $c\in \mu^{\kappa-\tika}S_{\delta}(\xim^{3\kappa-2})$ thanks to Lemmas \ref{lem:sanuki} and \ref{lem:yodobasi} one concludes that the right hand-side is $\simeq 0$. On the other hand note that
\begin{align*}
|({\rm Op}(e^{-\Lambda})\mD{\rm Op}(r)\mD u,{\rm Op}(e^{-\Lambda})Au)|\\
\leq C\|\mD{\rm Op}(r)\mD u\|
\|{\rm Op}(e^{-\Lambda})Au\|
\end{align*}
since $e^{-\Lambda}\in S_{\delta}(1)$ by Lemma \ref{lem:yoritune}. We apply Lemma \ref{lem:sanuki} to conclude that the right-hand side is $\simeq 0$. In conclusion we have
\begin{equation}
\label{eq:anman}
|({\rm Op}(e^{-\Lambda}){\rm Op}(R)u, {\rm Op}(e^{-\Lambda})Au)|\simeq 0.
\end{equation}
Finally we note that
\begin{align*}
-2{\mathsf{ Im}}\,(\mD^{\kappa} {\rm Op}(e^{-\La})Au,\mD^{\kappa} {\rm Op}(e^{-\La})u)\\
\leq \|\mD^{3\kappa/2}{\rm Op}(e^{-\La})u\|^2+\|\xim^{\kappa/2}{\rm Op}(e^{-\La})Au\|^2\preceq {\mathcal E}_1+{\mathcal E}_2
\end{align*}
and 
\begin{align*}
-2{\mathsf{ Im}}\,({\rm Op}(e^{-\La})P^{\sharp}u, {\rm Op}(e^{-\La})Au)
\leq   \|{\rm Op}(e^{-\Lambda})P^{\sharp}u\|^2+E
\end{align*}
by \eqref{eq:tiyo}.
From \eqref{eq:samui} and Propositions \ref{prva}, \ref{prvi} and  \ref{prvo} and Lemma \ref{levki} we conclude that
\begin{equation}
\label{eq:hayabusa}
\frac{dE}{dt}\leq -(2\theta-\tau\sqrt{c^*}-1-C\mu^{\epsilon})({\mathcal E}_1+{\mathcal E}_2)+E+\|{\rm Op}(e^{-\Lambda})P^{\sharp}u\|^2
\end{equation}
with some $\epsilon>0$. We now fix $\theta_0$ and $\mu>0$ such that 
\begin{equation}
\label{eq:yazu}
2\theta_0>T\sqrt{c^*}+1,\quad 2\theta_0-T\sqrt{c^*}-1-C\mu^{\epsilon}\geq 0.
\end{equation}
Then from \eqref{eq:hayabusa} it follows that for $0\leq t\leq \tau/\theta_0$
\[
E(u;t)\leq e^TE(u;0)+e^T\int_0^t\|{\rm Op}(e^{-\Lambda})P^{\sharp}u\|^2ds.
\]
Since $\Lambda(0)=0$ and $E(u;0)\leq \|(D_t-\theta_0\mD^{\kappa})u(0)\|^2+C\|\mD u(0)\|^2$ this shows that
\begin{eqnarray*}
\|\mD^{\kappa}{\rm Op}(e^{-\Lambda})u(t)\|^2+\|{\rm Op}(e^{-\Lambda})Au(t)\|^2
\leq C\Big\{\|\mD u(0)\|^2+\|D_tu(0)\|^2\Big\}\\[3pt]
+C\int_0^t\|{\rm Op}(e^{-\Lambda}){\rm Op}(e^{(\tau-\theta_0s)\xim^{\kappa}})P\,{\rm Op}(e^{-(\tau-\theta_0 s)\xim^{\kappa}})u(s)\|^2ds.
\end{eqnarray*}
Replacing $u$ by ${\rm Op}(e^{ (\tau-\theta_0 t)\xim^{\kappa}})u$ we obtain
\begin{proposition}
\label{pro:thva}
We have
\begin{eqnarray*}
\|\mD^{\kappa}{\rm Op}(e^{-\Lambda}){\rm Op}(e^{ (\tau-\theta_0t)\xim^{\kappa}})u(t)\|^2
+\|{\rm Op}(e^{-\Lambda}){\rm Op}(e^{(\tau-\theta_0 t)\xim^{\kappa}})Au(t)\|^2\nonumber\\
\leq C\Big\{\|\mD {\rm Op}(e^{ \tau\xim^{\kappa}})u(0)\|^2+\|{\rm Op}(e^{ \tau\xim^{\kappa}}) D_tu(0)\|^2\Big\}\\
+C\int_0^t\|{\rm Op}(e^{-\Lambda}){\rm Op}(e^{ (\tau-\theta_0s)\xim^{\kappa}})Pu(s)\|^2ds
\end{eqnarray*}
for $0\leq t\leq \tau/\theta_0$.
\end{proposition}
\begin{lemma}
\label{levku} For any $0<\tau'<\tau$ and $s\in\R$ there exists $C>0$  such that
\begin{align*}
\|\mD^se^{(\tau'-\theta_0t)\mD^{\kappa}}u\|\leq C\|\mD^s{\rm Op}(e^{-\Lambda})e^{(\tau-\theta_0t)\mD^{\kappa}}u\|\\
\leq C'\|\mD^se^{ (\tau-\theta_0t)\mD^{\kappa}}u\|
\end{align*}
for $0\leq t\leq \tau'/\theta_0$.
\end{lemma}
\begin{proof} Thanks to Lemma \ref{lem:piko} we have $e^{-(\tau-\tau')\xim^{\kappa}}\in S_{\delta}^{\lr{s}}(e^{-c\xim^{\kappa}})$ where $0<c<\tau-\tau'$. Denote 
\[
p=e^{-(\tau-\tau')\xim^{\kappa}}\#e^{\Lambda}
\]
then Proposition \ref{pro:kodaira} shows that $p\in S_{\delta}(\xim^{l})$ for any $l\in\R$.  Repeating the same arguments as in the proof of Lemma \ref{leiva} we have 
${\rm Op}(e^{-(\tau-\tau')\xim^{\kappa}})={\rm Op}(p){\rm Op}(K){\rm Op}(e^{-\Lambda})$ 
since $1={\rm Op}(e^{\Lambda}){\rm Op}(K){\rm Op}(e^{-\Lambda})$ and hence 
\[
\|{\rm Op}(e^{-(\tau-\tau')\xim^{\kappa}})u\|=\|{\rm Op}(p){\rm Op}(K){\rm Op}(e^{-\Lambda})u\|
\leq C\|{\rm Op}(e^{-\Lambda})u\|
\]
which proves the first inequality replacing $u$ by ${\rm Op}(e^{(\tau-\theta_0t)\xim^{\kappa}})\mD^su$ and taking  Corollary \ref{cor:koike} into account. On the other hand since $e^{-\La}\in S_{\delta}(1)$ by Lemma \ref{lem:yoritune} the second inequality is clear from Corollary \ref{cor:koike}. 
\end{proof}
Thanks to Lemma \ref{levku}  it follows from Proposition \ref{pro:thva} that
\begin{theorem}
\label{thm:thva}
For any $0<\tau'<\tau$ there exists $C>0$  such that one has
\begin{eqnarray*}
\sum_{j=0}^1\|\mD^{(1-j)\kappa}{\rm Op}(e^{(\tau'-\theta_0t)\xim^{\kappa}})D_t^ju(t)\|^2
\leq C\sum_{j=0}^1\|\mD^{1-j} {\rm Op}(e^{ \tau\xim^{\kappa}})D_t^ju(0)\|^2\\[3pt]
+C\int_0^t\|{\rm Op}(e^{ (\tau-\theta_0s)\xim^{\kappa}})Pu(s)\|^2ds
\end{eqnarray*}
for $0\leq t\leq \tau'/\theta_0$.
\end{theorem}
%

\section{Proof of Theorem \ref{thm:taifu}}
\label{sec:thmproof}

 We prove a detailed version of Theorem \ref{thm:taifu}.
 \begin{proposition}
 \label{pro:seikaku}Assume $0\leq b(t)\in C^{n,\al}([0,T])$ and $0\leq a(x)\in {\mathcal G}^s(\R)$. There exists $\theta_0>0$ such that for any $1<s<s'<1+(n+\al)/2$ and $0<\tau \,(\leq T)$ there is $\mu>0$  such that for any $u_j\in {\mathcal G}_0^{s'}(\R)$ with finite $\sum_{j=0}^1\|\mD^{1-j}{\rm Op}(e^{\tau\xim^{\kappa}})u_j\|^2$ and any $0<\tau'<\tau$ there exist
  $C>0$ and a unique $u\in C^1([0,\tau'/\theta_0];{\mathcal G}_0^{s'}(\R))$, solution to \eqref{eq:mondai}, such that
\begin{equation}
\label{eq:kihon}
\begin{split}
\sum_{j=0}^1\|\mD^{(1-j)\kappa}{\rm Op}(e^{(\tau'-\theta_0t)\xim^{\kappa}})D_t^ju(t)\|^2\\
\leq C\sum_{j=0}^1\|\mD^{1-j} {\rm Op}(e^{ \tau\xim^{\kappa}})D_t^ju(0)\|^2.
\end{split}
\end{equation}
Moreover if $\cup_{j=0}^1{\rm supp}\,u_j\subset \{|x| \leq R\}$ then ${\rm supp}\,u(t,\cdot)\subset \{|x|\leq R+{\hat c}t\}$.
 \end{proposition}
 \begin{proof} Take $\chi(x)\in {\mathcal G}_0^s(\R)$ such that $0\leq \chi(x)\leq 1$ and $1$ in $|x|\leq 1$ and $0$ for $|x|\geq 2$. We define $0\leq a_{\nu}(x)\in {\mathcal G}_0^s(\R)$ by $a_{\nu}(x)=a(x)\chi(\nu x)$. Then it is easy to see that there are $C>0$,  $A>0$ independent of $0<\nu\leq1$ such that
 \[
 |\dif_x^ka_{\nu}(x)|\leq CA^kk!^s\quad \forall x\in \R, \quad k=0,1,\ldots.
 \]
Since $C$ and $A$ are independent of $\nu$ it can be seen   that one can take $\theta_0$ and $\mu$ verifying  \eqref{eq:yazu} which are independent of $\nu$.  Introducing a small positive parameter $\varepsilon>0$ we set
\[
b_{\varepsilon}(t)=b(t)+\varepsilon,\quad a_{\nu,\varepsilon}(x)=a_{\nu}(x)+\varepsilon
\]
and consider $P_{\nu,\varepsilon}=D_t^2-D_x(b_{\varepsilon}(t)a_{\nu,\varepsilon}(x))D_x$.
Noting $a'_{\nu,\varepsilon}(x)=a_{\nu}'(x)$, $0\leq a_{\nu}(x)\leq a_{\nu,\varepsilon}(x)$ and $b'_{\varepsilon}(t)=b'(t)$, $0\leq b(t)\leq b_{\varepsilon}(t)$ it is not difficult to examine that Theorem \ref{thm:thva} holds uniformly in $0<\nu\leq 1$ and $0<\varepsilon\leq \epsilon_0$. Since $P_{\nu,\varepsilon}$ is strictly hyperbolic and $b_{\varepsilon}(t)a_{\nu,\varepsilon}(x)\leq {\hat c}^2$ for $(t,x)\in [0,T]\times \R$ with ${\hat c}>0$ independent of $0<\varepsilon\leq \varepsilon_0$ and $0<\nu\leq 1$ the Cauchy problem for $P_{\nu,\varepsilon}$ 
\[
\left\{\begin{array}{ll}
P_{\nu,\varepsilon}u_{\nu,\varepsilon}=0\;\;\mbox{in}\;\;(t,x)\in [0,T]\times \R,\\
D_t^ju_{\nu,\varepsilon}(0,x)=u_j(x)  \;\;\mbox{for}\;\;j=0,1
\end{array}
\right.
\]
with $u_j(x)\in {\mathcal G}_0^{s'}(\R)$ supported in $|x|\leq R$
has a unique solution $u_{\nu,\varepsilon}$ such that 
\[
{\rm supp}\, u_{\nu,\varep}(t,\cdot)\subset\{x\mid |x|\leq R+{\hat c}t\}.
\]
From Theorem \ref{thm:thva} we conclude that 
 \[
\sup_{0\leq t\leq \tau'/\theta_0}\sum_{j=0}^1 \|\mD^{(1-j)\kappa}{\rm Op}(e^{(\tau'-\theta_0t)\xim^{\kappa}})D_t^ju_{\nu,\varepsilon}(t)\|
\]
 is uniformly bounded in $0<\varepsilon\leq \varepsilon_0$ provided   $\sum_{j=0}^1\|\mD^{1-j} {\rm Op}(e^{ \tau\xim^{\kappa}})u_j\|$ is finite. Therefore by the standard argument one can prove that there exists $u_{\nu}\in C^1([0,\tau');{\mathcal G}_0^{s'}(\R))$ satisfying 
\[
\left\{\begin{array}{ll}
D_t^2u_{\nu}-D_x(b(t)a_{\nu}(x))D_xu_{\nu}=0\;\;\mbox{in}\;\;(t,x)\in [0,T]\times \R,\\
D_t^ju_{\nu}(0,x)=u_j(x)  \;\;\mbox{for}\;\;j=0,1
\end{array}
\right.
\]
and the energy estimate \eqref{eq:kihon}. 
Since $D_x(b(t)a_{\nu}(x))D_xu_{\nu}=D_x(b(t)a(x))D_xu_{\nu}$ taking  $\nu^{-1}>R+{\hat c}T$ then $u_{\nu}$ is a desired solution to the Cauchy problem \eqref{eq:mondai}.  
 
 It remains to prove the uniqueness. Let $0<{\hat T}<\tau/\theta_0$. Since $P^*=P$ we conclude from Proposition \ref{pro:seikaku} that the Cauchy problem for $P^*$ with reversed time direction has a solution with finite propagation speed, that is  there is $\mu>0$ such that for any $\phi_j(x)\in {\mathcal G}_0^{s'}(\R)$ with finite $\sum_{j=0}^1\|\mD^{1-j}{\rm Op}(e^{(\tau+1)\xim^{\kappa}})\phi_j\|^2$ there exists a solution $v\in C^1([0,{\hat T}];{\mathcal G}_0^{s'}(\R))$ to the adjoint Cauchy problem
\[
\left\{\begin{array}{ll}
P^*v=0\;\;\mbox{in}\;\;(t,x)\in [0,{\hat T}]\times \R,\\
D_t^jv({\hat T},x)=\phi_j(x)  \;\;\mbox{for}\;\;j=0,1
\end{array}
\right.
\]
 with finite propagation speed ${\hat c}$. Now assume that $u\in C^1([0,\tau'/\theta_0];{\mathcal G}_0^{s'}(\R))$ is a solution to \eqref{eq:mondai} with $u_j(x)=0$ and $T'=\tau'/\theta_0$. Then following the Holmgren's arguments we have
 \begin{align*}
0=\int_0^{{\hat T}}(Pu,v)dt=-i(D_tu({\hat T}),\phi_0)-i(u({\hat T}),\phi_1)+\int_0^{{\hat T}}(u,P^*v)dt\\
=-i(D_tu({\hat T}),\phi_0)-i(u({\hat T}),\phi_1).
 \end{align*}
 Since $\phi_j$ are arbitraly we conclude that $u({\hat T})=0$.
 \end{proof}
 %


\section{Appendix}

In this section we give the proofs of Propositions \ref{pro:weyl:1},  \ref{pro:comLam}, \ref{pro:abLam} and \ref{pro:kodaira}, where for the convenience of applications to  operators of the form \eqref{eqia}, the assertions are stated in one dimensional space, while the dimension is irrelevant we prove the statements in $n$ dimensional space.

\subsection{Composition $(be^{\psi})\# e^{-\psi}$}

\begin{proposition}
\label{pro:syu:AA}Assume $\psi\in S_{\delta}^{\lr{s}}(\xim^{\tika})$ and $b\in S^{\lr{s}}_{\delta}(\xim^m)$ where  $ 1-\delta>\tika$ and $s(1-\delta)<1$.  Then we have
\[
(be^{\psi})\#e^{-\psi}=b+{\tilde b}+R
\]
where ${\tilde b}\in  S^{\lr{s}}_{\delta}\big(\xim^{m-(1-\delta-\tika)}\big)$ and $
R\in S^{(\bas)}_{0,0}\big(e^{-c\xim^{\baka}}\big)$ 
with some $\bas>1$, $\baka>0$, $c>0$ such that $\bas \tika<1$ and $\baka>\tika$.
\end{proposition}
\begin{corollary}
\label{cor:piko}
Assume $\psi\in S_{\delta}^{\lr{s}}(\xim^{\tika})$ where  $ 1-\delta>\tika$ and $s(1-\delta)<1$.   Then one has with some $\tis>1$
\[
e^{\psi}\#e^{-\psi}-1\in  S^{\lr{\tis}}_{\delta}(\xim^{-(1-\delta-\tika)})\subset \mu^{1-\delta-\tika} S_{\delta}(1).
\]
\end{corollary}
\begin{proof}We apply Proposition \ref{pro:syu:AA} with $b=1$ to get $e^{\psi}\#e^{-\psi}-1=r+R$ with $r\in S^{\lr{s}}_{\delta}\big(\xim^{-(1-\delta-\tika)}\big)$ and $
R\in S^{(\bas)}_{0,0}\big(e^{-c\xim^{\baka}}\big)$. We note that if $R\in S_{0,0}^{(\bas)}(e^{-c\xim^{\baka}})$ with $c>0$ then for any $0<c'<c$ and $\delta\geq 0$ one has $R\in S_{\delta}^{\lr{\tis}}(e^{-c'\xim^{\baka}})$ with some $\tis>1$. Indeed since
\begin{align*}
e^{-c\xim^{\baka}}\leq\xim^{-|\al|}\xim^{-\delta|\al+\be|}\xim^{(1+\delta)|\al+\be|}e^{-(c-c')\xim^{\baka}}e^{-c'\xim^{\baka}}\\
\leq CA^{|\al+\be|}\xim^{-|\al|}\xim^{-\delta|\al+\be|}|\al+\be|^{(1+\delta)(|\al+\be|)/\baka}e^{-c'\xim^{\baka}}
\end{align*}
it follows that
\begin{align*}
|\al+\be|^{\bas|\al+\be|}e^{-c\xim^{\baka}}\leq CA^{|\al+\be|}\big(|\al+\be|^{\bas+(1+\delta)/\baka}\xim^{-\delta}\big)^{|\al+\be|}
\xim^{-|\al|}e^{-c'\xim^{\baka}}
\end{align*}
that is $R\in S_{\delta}^{\lr{\tis}}(e^{-c'\xim^{\baka}})$ with $\tis=\bas+(1+\delta)/\baka$. Since $S_{\delta}^{\lr{\tis}}(\xim^{-(1-\delta-\tika)})\subset S_{\delta}(\xim^{-(1-\delta-\tika)})$ and $\xim^{-(1-\delta-\tika)}\leq \mu^{1-\delta-\tika}$  we get the assertion.
\end{proof}
%

\noindent
Proof of Proposition \ref{pro:syu:AA}: The idea of the proof is same as the proof of \cite[Theorem 5.1]{N:ojm}.
Consider  
\begin{align*}
(be^{\psi})\# e^{-\psi}=\int e^{-2i\sigma(Y,Z)}b(X+Y)e^{\psi(X+Y)-\psi(X+Z)}dYdZ\\
=b+\int e^{-2i\sigma(Y,Z)}b(X+Y)\big(e^{\psi(X+Y)-\psi(X+Z)}-1\big)dYdZ
\end{align*}
where $Y=(y,\eta)$, $Z=(z,\zeta)$ and $\sigma(Y,Z)=\lr{\eta, z}-\lr{y, \zeta}$. After the change of variables $Z\to Z+Y$ the integral on the right-hand side turns to 
\begin{equation}
\label{eq:sitamati}
\int e^{-2i\sigma(Y,Z)}b(X+Y)(e^{\psi(X+Y)-\psi(X+Y+Z)}-1)dYdZ.
\end{equation}
For any small $\varep>0$ we choose $\chi(r)\in \gamma^{(1+\varep)}(\R)$ which is $1$ in $|r|\leq 1/4$ and $0$ for $|r|\geq 1/2$.  
Denoting ${\hat\chi}=\chi(\lr{\eta}\xim^{-1})\chi(\lr{\zeta}\xim^{-1})$ and ${\tilde \chi}=\chi(|y|/6)\chi(\xim^{\delta}|z|/6)$ we note that
\begin{align*}
|\dif^{\be}_{(x,y,z)}\dif_{(\xi,\eta,\zeta)}^{\al}({\hat\chi}{\tilde\chi})|\leq CA^{|\al+\be|}|\al+\be|^{(1+\varep)|\al+\be|}
\xim^{-|\al|+\delta|\be|}.
\end{align*}
Write $1={\tilde\chi}{\hat \chi}+(1-{\tilde\chi}){\hat \chi}+(1-{\hat\chi})={\hat \chi}_0+{\hat \chi}_1+(1-{\hat \chi})$ and we first consider 
\begin{equation}
\label{eq:hanami}
\int e^{-2i\sigma(Y,Z)}b(X+Y)\big(e^{\psi(X+Y)-\psi(X+Y+Z)}-1\big){\hat\chi}_0dYdZ.
\end{equation}
 Inserting the Taylor expansion 
\begin{align*}
e^{\psi(X+Y)-\psi(X+Y+Z)}-1=-\sum_{i} \int_0^1 \big(\dif_{x_i}\psi(X+Y+\theta Z)z_i\\
+\dif_{\xi_i}\psi(X+Y+\theta Z)\zeta_i\big)e^{\psi(X+Y)-\psi(X+Y+\theta Z)}d\theta
\end{align*}
into \eqref{eq:hanami} then after  integration by parts we have
\begin{equation}
\label{eq:tako}
\begin{split}
\int e^{-2i\sigma(Y,Z)}b(X+Y)\big(e^{\psi(X+Y)-\psi(X+Y+Z)}-1\big){\hat \chi}_0dYdZ\\
=\int_0^1d\theta \int e^{-2i\sigma(Y,Z)}{\tilde b}(X,Y,Z)e^{\psi(X+Y)-\psi(X+Y+\theta Z)}dYdZ.
\end{split}
\end{equation}
\begin{lemma}
\label{lem:iike}Notations being as above. We have
\begin{align*}
|\dif_{(x,y)}^{\be}\dif_{(\xi,\eta)}^{\al}{\tilde b}(X,Y,Z)|\leq CA^{|\al+\be|}(|\al+\be|^{1+\varep}+|\al+\be|^s\xim^{-\delta})^{|\al+\be|}\\
\times \xim^{\tika}(1+\xim^{\tika}g_X^{1/2}(Z))\xim^{m-1+\delta}\xim^{-|\al|+\delta|\be|}.
\end{align*}
\end{lemma}
\begin{proof}Note that ${\tilde b}$ consists of terms such as
\begin{align*}
&{\tilde b}_1=\dif_y^{e}\dif_{\eta}^{f}\big(b(X+Y)\dif_{\xi}^{e}\dif_x^{f}\psi(X+Y+\theta Z){\hat \chi}_0\big),\\
&{\tilde b}_2=b(X+Y)\dif_x^{e}\dif_{\xi}^{f}\psi(X+Y+\theta Z)\\
&\qquad \quad\times
\big(\dif_{\xi}^{e}\dif_x^{f}\psi(X+Y)-\dif_{\xi}^{e}\dif_x^{f}\psi(X+Y+\theta Z)\big){\hat \chi}_0
\end{align*}
where $|e+f|=1$. The assertion for ${\tilde b}_1$ is easy. As for ${\tilde b}_2$ write
\begin{align*}
\dif_x^{\be}\dif_{\xi}^{\al}\big(\dif_{\xi}^{e}\dif_x^{f}\psi(X+Y)-\dif_{\xi}^{e}\dif_x^{f}\psi(X+Y+\theta Z)\big)\\
=\sum_j\int_0^1\big(\dif_x^{\be+f+e_j}\dif_{\xi}^{\al+e}\psi(X+ Y+s\theta Z)z_j\\
+\dif_x^{\be+f}\dif_{\xi}^{\al+e+e_j}\psi(X+Y+s\theta Z)\zeta_j\big)ds
\end{align*}
then on the support of ${\hat\chi}_0$ it is easy to see
\begin{align*}
|\dif_x^{\be}\dif_{\xi}^{\al}\big(\dif_{\xi}^{e}\dif_x^{f}\psi(X+Y)-\dif_{\xi}^{e}\dif_x^{f}\psi(X+Y+Z)\big)|\\
\leq CA^{|\al+\be|+2}(|\al+\be|+2+(|\al+\be|+2)^s\xim^{-\delta})^{|\al+\be|+2}\\
\times \xim^{\tika}\big(\xim^{\delta(|\be+f|+1)-|\al+e|}|z|+\xim^{\delta(|\be+f|-(|\al+e|+1)}|\zeta|\big)\\
\leq C_1A_1^{|\al+\be|}(|\al+\be|+|\al+\be|^s\xim^{-\delta})^{|\al+\be|}\xim^{\delta|\be+f|-|\al+e|}\xim^{{\tilde\kappa}}g_X^{1/2}(Z).
\end{align*}
Thus we get the assertion for ${\tilde b}_2$.
\end{proof}
%
%
\begin{lemma}
\label{lem:sakura} Let $\Psi(X,Y,Z)=\psi(X+Y)-\psi(X+Y+\theta Z)$ then  
on the support of ${\hat \chi}_0$ one has  
\[
|\Psi(X,Y,Z)|
\leq C\xim^{{\tilde\kappa}}g_X^{1/2}(Z)
\]
and
\begin{align*}
|\dif_{(x,y)}^{\be}\dif_{(\xi,\eta)}^{\al}e^{\Psi}|&\leq CA^{|\al+\be|}
\xim^{-|\al|+\delta|\be|}\\
&\times\big(\xim^{\tika}g_X^{1/2}(Z)+|\al+\be|+|\al+\be|^s\xim^{-\delta}\big)^{|\al+\be|}e^{\Psi}.
\end{align*}
\end{lemma}
\begin{proof}It suffices to repeat the proof of Lemma  \ref{lem:seikei}.
\end{proof}
Introducing the following differential operators and symbols
\[
\left\{\begin{array}{ll}
L=1+4^{-1}\xim^{2}|D_{\eta}|^2+4^{-1}\xim^{-2\delta}|D_y|^2,\\[3pt]
M=1+4^{-1}\xim^{2\delta}|D_{\zeta}|^2+4^{-1}\xim^{-2}|D_z|^2,\\[3pt]
\Phi=1+\xim^{2}|z|^2+\xim^{-2\delta}|\zeta|^2
=1+\xim^{2(1-\delta)}g_X(Z),\\[3pt]
\Theta=1+\xim^{2\delta}|y|^2+\xim^{-2}|\eta|^2=1+g_X(Y)\end{array}\right.
\]
so that $\Phi^{-N}L^Ne^{-2i\sigma(Y,Z)}=e^{-2i\sigma(Y,Z)}$ and $\Theta^{-\ell}M^{\ell}e^{-2i\sigma(Y,Z)}=e^{-2i\sigma(Y,Z)}$ we make integration by parts in \eqref{eq:sitamati}. Let $
F={\tilde b}(X,Y,Z)e^{\Psi}$ where ${\tilde b}$ is given in \eqref{eq:tako} and consider 
\begin{equation}
\label{eq:iseki}
\dif_x^{\be}\dif_{\xi}^{\al}\! \int e^{-2i\sigma(Y,Z)}FdYdZ
=\int e^{-2i\sigma(Y,Z)}L^N\Phi^{-N} M^{\ell}\Theta^{-\ell}(\dif_x^{\be}\dif_{\xi}^{\al}F)dYdZ.
\end{equation}
Here we note that $
|\dif_{(\xi,\eta)}^{\al}\dif_y^{\be}\Theta^{-\ell}|\leq CA^{|\al+\be|}|\al+\be|!\xim^{-|\al|+\delta|\be|}\Theta^{-\ell}$. 
Applying Lemmas \ref{lem:iike} and  \ref{lem:sakura} we can estimate $\big|L^N \Phi^{-N} M^{\ell}\Theta^{-\ell}(\dif_x^{\be}\dif_{\xi}^{\al}F)\big|
$ by
\begin{equation}
\label{eq:hayas}
\begin{split}
 C_{\ell}A_1^{2N+|\al+\be|}\Phi^{-N}\Theta^{-\ell}\xim^{-(1-\delta -\tika)}(1+\xim^{\tika}g_X^{1/2}(Z))
\big(\xim^{\tika}g_X^{1/2}(Z)\\
+(2N+|\al+\be|)^{1+\varep}
+(2N+|\al+\be|)^s\xim^{-\delta})^{2N+|\al+\be|}e^{\Psi}\xim^{m-|\al|+\delta|\be|}.
\end{split}
\end{equation}
Here we recall the following easy lemma \cite[Lemma 5.3]{N:ojm}.
\begin{lemma}
\label{lem:kanta}
Let $A\geq 0$, $B\geq 0$. Then there exists $C>0$ independent of $n, m\in\N$, $A$, $B$ such that
\begin{align*}
(A+(n+m)^{1+\varep}+(n+m)^sB)^{n+m}\\
\leq C^{n+m}(A+n^{1+\varep}+n^sB)^n(A+m^{1+\varep}+m^sB)^m.
\end{align*}
\end{lemma}
From Lemma \ref{lem:kanta} the right-hand side  of \eqref{eq:hayas} can be estimated by
\begin{align*}
 C_{\ell}A_1^{2N+|\al+\be|}\Phi^{-N}\Theta^{-\ell}
\xim^{-(1-\delta-\tika)}(1+\xim^{\tika}g_X^{1/2}(Z))\\
\times \big(\xim^{\tika}g_X^{1/2}(Z)
+(2N)^{1+\varep}+(2N)^s\xim^{-\delta}\big)^{2N}\\
\times 
 \big(\xim^{\tika}g_X^{1/2}(Z)
+|\al+\be|^{1+\varep}+|\al+\be|^s\xim^{-\delta}\big)^{|\al+\be|} 
 \xim^{m-|\al|+\delta|\be|}.
\end{align*}
Writing
\begin{equation}
\label{eq:bove}
\begin{split}
A_1^{2N}\Phi^{-N}\big(\xim^{\tika}g_X^{1/2}(Z)+(2N)^{1+\varep}+(2N)^s\xim^{-\delta}\big)^{2N}\\
=\Big(\frac{A_1\xim^{\tika}g_X^{1/2}(Z)}{\Phi^{1/2}}+\frac{A_1(2N)^{1+\varep}}{\Phi^{1/2}}+\frac{A_1(2N)^s\xim^{-\delta}}{\Phi^{1/2}}\Big)^{2N}
\end{split}
\end{equation}
we choose $N=N(Z,\xi)$ so that $
2N={\bar c}\,\Phi^{1/2(1+\varep)}$ with a small ${\bar c}>0$. Assuming $1+\varep<s$ without restrictions  one has
\begin{align*}
\Phi^{s/2(1+\varep)-1/2}\xim^{-\delta}=(1+\xim^{2(1-\delta)}g_X(Z))^{s/2(1+\varep)-1/2}\xim^{-\delta}\\
\leq C\xim^{(s/(1+\varep)-1)(1-\delta)-\delta}
\leq C\xim^{s(1-\delta)/(1+\varep)-1}\leq C\mu^{c}
\end{align*}
with some $c>0$ on the support of ${\hat\chi}_0$ because $g_X(Z)$ is bounded there and $s(1-\delta)<1$.  Then the right-hand side of \eqref{eq:bove} is bounded by 
\begin{align*}
 \big(A_1(\mu^{1-\delta+\tika} +{\bar c}^{1+\varep}+\mu^{c})\big)^{2N}\leq Ce^{-c_1\Phi^{1/2(1+\varep)}}
\end{align*}
choosing ${\bar c}$ and  $\mu>0$ small such that $A_1(\mu^{1-\delta+\tika} +{\bar c}^{1+\varep}+\mu^{c})<1$. 
Assuming $\varep>0$ small so that $1-\delta>(1+\varep)\tika$ then since $g_X^{1/2}(Z)\leq Cg_X^{1/2(1+\varep)}(Z)$ on the support of ${\hat\chi}_0$ we have 
\begin{equation}
\label{eq:tetu}
\Phi^{1/2(1+\varep)}\geq c\xim^{(1-\delta)/(1+\varep)}g_X^{1/2}(Z)\geq c\mu^{-(1-\delta)/(1+\varep)}\xim^{\tika}g_X^{1/2}(Z).
\end{equation}
Then one has
\begin{align*}
(\xim^{\tika}g_X^{1/2}(Z)
+|\al+\be|^{1+\varep}+|\al+\be|^s\xim^{-\delta})^{|\al+\be|}e^{-c\,\Phi^{1/2(1+\varep)}}\\
\leq CA^{|\al+\be|}(|\al+\be|^{1+\varep}+|\al+\be|^s\xim^{-\delta})^{|\al+\be|}e^{-c'\Phi^{1/2(1+\varep)}}.
\end{align*}
Noting $e^{-c'\Phi^{1/2(1+\varep)}}\leq C_{\ell}\Phi^{-\ell}$ and Lemma \ref{lem:sakura} and \eqref{eq:tetu} we have
\begin{equation}
\label{eq:itiban:a}
\begin{split}
|L^N \Phi^{-N}
M^{\ell} \Theta^{-\ell}(\dif_x^{\be}\dif_{\xi}^{\al}F)|
\leq 
\mu^{\tika} C_{\ell}A^{|\al+\be|}(|\al+\be|^{1+\varep}\\
 +|\al+\be|^s\xim^{-\delta})^{|\al+\be|}
 \xim^{-(1-\delta-\tika)}
\xim^{m-|\al|+\delta|\be|}\Theta^{-\ell}\Phi^{-\ell}.
 \end{split}
\end{equation}
Finally choosing $\ell>(n+1)/2$ and recalling $\int \Theta^{-\ell}\Phi^{-\ell}dYdZ=C$  we conclude that
\begin{equation}
\label{eq:suika:a}
\begin{split}
\Big|\dif_x^{\be}\dif_{\xi}^{\al}\int e^{-2i\sigma(Y,Z)}F{\hat \chi}_0dYdZ\Big|\\
\leq  CA^{|\al+\be|}(|\al+\be|^{1+\varep}+|\al+\be|^s\xim^{-\delta})^{|\al+\be|}\xim^{-(1-\delta-\tika)}\xim^{m-|\al|+\delta|\be|}.
\end{split}
\end{equation}

 Denoting $F=b(X+Y)(e^{\psi(X+Y)-\psi(X+Y+Z)}-1){\hat \chi}_1$ we next consider 
\begin{equation}
\label{eq:suika:b}
\begin{split}
\dif_x^{\be}\dif_{\xi}^{\al}\int e^{-2i\sigma(Y,Z)}FdYdZ\\
=
\int e^{-2i\sigma(Y,Z)}(\xim^{2\delta}|D_{\eta}|^2+|D_{\zeta}|^2)^N (\xim^{2\delta}|z|^2+|y|^2)^{-N} \dif_x^{\be}\dif_{\xi}^{\al}FdYdZ.
\end{split}
\end{equation}
Since
$|\psi(X+Y)|+|\psi(X+Y+Z)|$ is bounded by $C\xim^{\tika}$  and $C^{-1}\leq \lr{\xi+\eta}_{\mu}/\xim$, $ \lr{\xi+\zeta}_{\mu}/\xim\leq C$  with some $C>0$ on the support of ${\hat\chi}_1$ thanks to Lemma \ref{lem:sakura} it is not difficult to show
\begin{align*}
\big|(\xim^{2\delta}|D_{\eta}|^2+|D_{\zeta}|^2)^N\dif_x^{\be}\dif_{\xi}^{\al}F\big|
\leq CA^{2N+|\al+\be|}\xim^{m-|\al|+\delta|\be|}\\
\times (C\xim^{\tika}+|\al+\be|^{1+\varep}+|\al+\be|^s\xim^{-\delta})^{|\al+\be|}
\\
 \times  (C\xim^{\tika}+(2N)^{1+\varep}+(2N)^s\xim^{-\delta})^{2N} \xim^{-2(1-\delta) N}e^{c\xim^{\tika}}.
\end{align*}
Choose $2N=c_1 \xim^{(1-\delta)/(1+\varep)}$ with small $c_1>0$ so that
\begin{align*}
A^{2N}\xim^{-2(1-\delta) N}\big(C \xim^{\tika}+(2N)^{1+\varep}+(2N)^s\xim^{-\delta}\big)^{2N}\\
=\Big(\frac{AC\xim^{\tika}}{\xim^{1-\delta}}+Ac_1^{1+\varep}+\frac{Ac_1^s\xim^{s(1-\delta)/(1+\varep)}}{\xim}\Big)^{2N}
\end{align*}
is bounded by $Ce^{-c\xim^{(1-\delta)/(1+\varep)}}$ choosing $\mu$ small which is possible because $\tika<1-\delta$ and $s(1-\delta)<1$. Note that $\xim^{\tika |\al+\be|}e^{-c\xim^{(1-\delta)/(1+\varep)}}$ is bounded by 
\begin{align*}
CA^{|\al+\be|} \big(|\al+\be|^{(1+\varep)\tika/(1-\delta)}\big)^{|\al+\be|}e^{-c'\xim^{(1-\delta)/(1+\varep)}}\\
\leq CA^{|\al+\be|}|\al+\be|^{(1+\varep)|\al+\be|}e^{-c'\xim^{(1-\delta)/(1+\varep)}}
\end{align*}
for $\tika/(1-\delta)<1$ where $0<c'<c$. Since $|y|^2+\xim^{2\delta}|z|^2\geq 1$ if ${\hat\chi}_1\neq 0$ then
\[
\xim^{-2n+n\delta}\int (|y|^2+\xim^{2\delta}|z|^2)^{-N}{\hat\chi}_1dYdZ\leq C.
\]
Noting $\xim^{2n-n\delta}e^{-c'\xim^{(1-\delta)/(1+\varep)}}\leq C$ we conclude that \eqref{eq:suika:b} is also bounded by the right-hand side of \eqref{eq:suika:a}: Therefor recalling ${\hat \chi}={\hat\chi}_0+{\hat\chi}_1$ we have
\begin{lemma}
\label{lem:syubu} Let ${\hat\chi}=\chi(\lr{\eta}\xim^{-1})\chi(\lr{\zeta}\xim^{-1})$. Then  for any $\varep>0$ there exist $C>0,A>0$ such that we have 
\begin{align*}
\Big|\dif_x^{\be}\dif_{\xi}^{\al}\int e^{-2i\sigma(Y,Z)}b(X+Y)(e^{\psi(X+Y)-\psi(X+Y+Z)}-1){\hat\chi}dYdZ\Big|\\
\leq  CA^{|\al+\be|}(|\al+\be|^{1+\varep}+|\al+\be|^{s}\xim^{-\delta}\big)^{|\al+\be|}
 \xim^{-(1-\delta-\tika)}\xim^{m-|\al|+\delta|\be|}.
\end{align*}
\end{lemma}
We turn to $\int e^{-2i\sigma(Y,Z)}b(X+Y)(e^{\psi(X+Y)-\psi(X+Y+Z)}-1)(1-{\hat\chi})dYdZ$. Making the change of variables $Z\to Z-Y$ we come back to the original coordinates. Denote $\chi(\xim^{-1}\lr{\eta})\chi(\xim^{-1}\lr{\zeta-\eta})={\tilde\chi}_1$ and write
\begin{align*}
1-{\tilde\chi}_1=(1-{\tilde\chi}_1)\chi(\lr{\eta}^{-1}\lr{\zeta})+(1-{\tilde\chi}_1)(1-\chi(\lr{\eta}^{-1}\lr{\zeta}))
={\hat\chi}_2+{\hat\chi}_3.
\end{align*}
Note that on the support of ${\hat\chi}_2$ we have $C\lr{\eta}\geq \lr{\zeta}$ and $C\lr{\eta}\geq \xim$. Similarly on the support of ${\hat\chi}_3$ one has $C\lr{\zeta}\geq \lr{\eta}$ and $C\lr{\zeta}\geq \xim$.
Denoting $F=b(X+Y)(e^{\psi(X+Y)-\psi(X+Z)}-1)$ we consider  
\begin{align*}
\dif_x^{\be}\dif_{\xi}^{\al}\int e^{-2i\sigma(Y,Z)}F{\hat\chi}_idYdZ=\int e^{-2i\sigma(Y,Z)}\lr{\eta}^{-2N_2}\lr{\zeta}^{-2N_1}\\
\times \lr{D_z}^{2N_2}\lr{D_y}^{2N_1}
 \lr{y}^{-2\ell}\lr{z}^{-2\ell}\lr{D_{\zeta}}^{2\ell}\lr{D_{\eta}}^{2\ell}\dif_x^{\be}\dif_{\xi}^{\al}F{\hat\chi}_idYdZ.
\end{align*}
For case ${\hat\chi}_2$ we choose $N_1=\ell$, $N_2=N$. Since $|\psi(X+Y)|+|\psi(X+Z)|\leq C\lr{\eta}^{\tika}$  on the support of ${\hat\chi}_2$ it is not difficult to see that 
\[
|\lr{\eta}^{-2N}\lr{\zeta}^{-2\ell}\lr{D_z}^{2N}\lr{D_y}^{2\ell}\lr{y}^{-2\ell}\lr{z}^{-2\ell}\lr{D_{\zeta}}^{2\ell}\lr{D_{\eta}}^{2\ell}\dif_x^{\be}\dif_{\xi}^{\al}F{\hat\chi}_2|
\]
is bounded by
\begin{equation}
\label{eq:copii}
\begin{split}
C_{\ell}A^{2N+|\al+\be|}\lr{\eta}^{-2N}\lr{\zeta}^{-2\ell}\lr{y}^{-2\ell}\lr{z}^{-2\ell}\lr{\eta}^{|m|+2\delta\ell+6\tika\ell}\\
\times (C\lr{\eta}^{\tika+\delta}+N^{1+\varep}\lr{\eta}^{\delta}+N^s)^{2N}
 (C\lr{\eta}^{\tika+\delta}\\
+|\al+\be|^{1+\varepsilon}\lr{\eta}^{\delta} +|\al+\be|^s)^{|\al+\be|}
e^{C\lr{\eta}^{\tika}}
\end{split}
\end{equation}
because 
\begin{equation}
\label{eq:komaba}
\begin{split}
(C\lr{\omega}^{\tika}+(2k)^{1+\varep}+(2k)^s\lr{\omega}^{-\delta})^{2k}\lr{\omega}^{2 \delta l}\\
\leq (C\lr{\omega}^{\tika+\delta}+(2k)^{1+\varep}\lr{\omega}^{\delta}+(2k)^s)^{2k}
\end{split}
\end{equation}
for any $l, k\in\N$, $l\leq k$ and any $\omega$, in particular $\omega=\xi+\eta$, $\omega=\xi+\zeta$.  
Here writing 
\begin{align*}
A^{2N}\lr{\eta}^{-2N}\big( C\lr{\eta}^{\tika+\delta}+(2N)^{1+\epsilon}\lr{\eta}^{\delta}+(2N)^s\big)^{2N}\\
=\Big(\frac{AC \lr{\eta}^{\tika+\delta}}{\lr{\eta}}+\frac{A(2N)^{1+\varep}}{\lr{\eta}^{1-\delta}}+\frac{A(2N)^s}{\lr{\eta}}\Big)^{2N}
\end{align*}
we take $
2N=c_1 \lr{\eta}^{(1-\delta)/(1+\varep)}$ 
with small $c_1>0$ so that the right-hand side is bounded by $Ce^{-c\lr{\eta}^{(1-\delta)/(1+\varep)}}$ with some $c>0$ since $s(1-\delta)<1$. Since
\begin{equation}
\label{eq:pean}
\begin{split}
&\lr{\eta}^{\delta |\al+\be|}e^{-c\lr{\eta}^{(1-\delta)/(1+\varep)}}\\
&\leq CA_1^{|\al+\be|}\big(|\al+\be|^{\delta(1+\varep)/(1-\delta)}\big)^{|\al+\be|}e^{-c_1\lr{\eta}^{(1-\delta)/(1+\varep)}},\\
 &\lr{\eta}^{(\tika+\delta) |\al+\be|}e^{-c\lr{\eta}^{(1-\delta)/(1+\varep)}}\\
 &\leq CA_1^{|\al+\be|}\big(|\al+\be|^{(1+\varep)(\tika+\delta)/(1-\delta)}\big)^{|\al+\be|}e^{-c_1\lr{\eta}^{(1-\delta)/(1+\varep)}}
 \end{split}
 \end{equation}
  ($0<c_1<c$) and 
 \begin{equation}
 \label{eq:pean:b}
\frac{(1+\varep)\delta}{1-\delta}+1+\varepsilon=\frac{1+\varep}{1-\delta},\quad
 \frac{(1+\varep)(\tika+\delta)}{1-\delta}< \frac{1+\varep}{1-\delta}
 \end{equation}
 one sees that \eqref{eq:copii} is bounded by
\begin{align*}
C_{\ell}A_1^{|\al+\be|}\lr{\zeta}^{-2\ell}\lr{y}^{-2\ell}\lr{z}^{-2\ell}(|\al+\be|^{(1+\varep)/(1-\delta)}\big)^{|\al+\be|}
 e^{-c_1\lr{\eta}^{(1-\delta)/(1+\varep)}}.
\end{align*}
For case ${\hat\chi}_3$ choosing $N_1=N$ and $N_2=\ell$  we can  prove that 
\[
\big|\lr{\eta}^{-2\ell}\lr{\zeta}^{-2N}\lr{D_z}^{2\ell}\lr{D_y}^{2N}\lr{y}^{-2\ell}\lr{z}^{-2\ell}\lr{D_{\zeta}}^{2\ell}\lr{D_{\eta}}^{2\ell}\dif_x^{\be}\dif_{\xi}^{\al}F{\hat\chi}_3\big|
\]
 is bounded by 
 \begin{align*}
 C_{\ell}A^{2N+|\al+\be|}\lr{\eta}^{-2\ell}\lr{\zeta}^{-2N}\lr{y}^{-2\ell}\lr{z}^{-2\ell}\lr{\zeta}^{|m|+2\delta\ell+6\tika \ell}\\
\times (C\lr{\zeta}^{\tika+\delta}+N^{1+\varep}\lr{\zeta}^{\delta}+N^s)^{2N}
 (C\lr{\zeta}^{\tika+\delta}\\
+ |\al+\be|^{1+\varep}\lr{\zeta}^{\delta}+|\al+\be|^s)^{|\al+\be|}
e^{C\lr{\zeta}^{\tika}}.
 \end{align*}
 Therefore repeating the same arguments, choosing $
2N=c_1 \lr{\zeta}^{(1-\delta)/(1+\varep)}$ now,  this is bounded by
\begin{align*}
C_{\ell}A_1^{|\al+\be|}\lr{\eta}^{-2\ell}\lr{y}^{-2\ell}\lr{z}^{-2\ell}(|\al+\be|^{(1+\varep)/(1-\delta)})^{|\al+\be|}
e^{-c_1\lr{\zeta}^{(1-\delta)/(1+\varep)}}.
\end{align*}
Thus recalling that $\xim\leq C\lr{\eta}$, $\xim\leq  C\lr{\zeta}$ on the support of ${\hat\chi}_2$,  ${\hat\chi}_3$ we get
\begin{lemma}
\label{lem:syu:b}We have
\begin{align*}
\Big|\dif_x^{\be}\dif_{\xi}^{\al}\int e^{-2i\sigma(Y,Z)}b(X+Y)(e^{\psi(X+Y)-\psi(X+Z)}-1)(1-{\tilde\chi}_1)dYdZ\Big|\\
\leq  CA^{|\al+\be|}(|\al+\be|^{(1+\varep)/(1-\delta)})^{|\al+\be|}e^{-c_1\xim^{(1-\delta)/(1+\varep)}}.
\end{align*}
\end{lemma}
 We choose ${\bar\varepsilon}>0$ such that $1-\delta>(1+{\bar\varepsilon})\tika$  and set $\baka=(1-\delta)/(1+{\bar\varepsilon})$. With $\bas=(1+{\bar\varepsilon})/(1-\delta)$ ($\bas>s$) we finish the proof of Proposition \ref{pro:syu:AA}.
%

\subsection{Composition $(be^{\psi})\#a$ }

\begin{proposition}
\label{pro:apsiomega} Assume $\psi\in S_{\delta}^{\lr{s}}(\xim^{\tika})$ and $a\in S^{\lr{s}}_{\delta}(\xim^{m_2})$, $b\in S^{\lr{s}}_{\delta}(\xim^{m_1})$ where $1-\delta>\tika$ and $(1-\delta)s<1$. Then for any $p\in \N$ we have
\[
(be^{\psi})\#a=\sum_{|\al+\be|<p}\frac{(-1)^{|\be|}}{(2i)^{|\al+\be|}\al!\be!}a^{(\be)}_{(\al)}(be^{\psi})^{(\al)}_{(\be)}+r_pe^{\psi}+R_p
\]
where $r_p\in \mu^{\tika p} S^{\lr{s}}_{\delta}\big(\xim^{m_1+m_2-(1-\delta -\tika)p}\big)$  and $R_p\in S_{0,0}^{(\bar s)}(e^{-c\xim^{\baka}})$ 
with some ${\bar s}>1$, $\baka$ and $c>0$  satisfying $\bas \tika<1$ and $\baka>\tika$.
 For $a\#(be^{\psi})$ similar assertion holds, where $(-1)^{|\be|}$ is replaced by $(-1)^{|\al|}$.
\end{proposition}
\noindent
Proof of Proposition \ref{pro:apsiomega}:
Write
\begin{align*}
(b e^{\psi})\# a=\int e^{-2i\sigma(Y,Z)}b(X+ Y)e^{\psi(X+Y)}a(X+Z)dYdZ
\end{align*}
with $Y=(y,\eta)$, $Z=(z,\zeta)$ and replace $a(X+Z)$ by its Taylor expansion
\begin{align*}
\sum_{|\al|<p}\frac{1}{\al!}\dif_X^{\al}a(X)Z^{\al}
+\sum_{|\al|=p}\frac{p}{\al!}\int_0^1(1-\theta)^{p-1}\dif_X^{\al}a(X+\theta Z)d\theta\cdot Z^{\al}\\
=\sum_{|\al|<p}\frac{1}{\al!}\dif_X^{\al}a(X)Z^{\al}+R_p
\end{align*}
where
\[
R_p=\sum_{|\al|=p}\frac{p}{\al!}\int_0^1(1-\theta)^{p-1}\dif_X^{\al}a(X+\theta Z))d\theta \cdot Z^{\al}.
\]
Here note that the first term on the left-hand side yeilds
\begin{equation}
\label{eq:yunohi}
\begin{split}
\sum_{|\al|<p}\frac{1}{\al!}\int e^{-2i\sigma(Y,Z)}b(X+Y)e^{\psi(X+Y)}\dif_X^{\al}a(X)Z^{\al}dYdZ\\
=\sum_{|\mu+\nu|<p}\frac{(-1)^{|\nu|}}{(2i)^{|\mu+\nu|}\mu!\nu!}\big(\dif_{\xi}^{\mu}\dif_x^{\nu}(b(X)e^{\psi(X)})\big)\dif_x^{\mu}\dif_{\xi}^{\nu}a(X).
\end{split}
\end{equation}
Consider the remainder term
\begin{align*}
\int e^{-2i\sigma(Y,Z)}b(X+ Y)e^{\psi(X+Y)}R_pdYdZ
=p\sum_{|\mu+\nu|=p}C_{\mu,\nu}\int_0^1(1-\theta)^{p-1}\\
\times \int e^{-2i\sigma(Y,Z)}\dif_{\eta}^{\mu}\dif_y^{\nu}\big(b(X+ Y)e^{\psi(X+Y)}\big) \dif_{\xi}^{\nu}\dif_x^{\mu}a(X+\theta Z)d\theta dYdZ.
\end{align*}
with $
C_{\mu,\nu}=(-1)^{|\nu|}/((2i)^{|\mu+\nu|}\mu!\nu!)$. 
Denoting 
\[
F=\dif_y^{\nu}\dif_{\eta}^{\mu}\big(b(X+Y)e^{\psi(X+Y)}\big)\dif_{\xi}^{\nu}\dif_x^{\mu}a(X+\theta Z)
\]
and ${\hat\chi}=\chi(\lr{\eta}\xim^{-1})\chi(\lr{\zeta}\xim^{-1})$, ${\tilde \chi}=\chi(\xim^{\delta}|y|/6)\chi(|z|/6)$ we write $\int e^{-2i\sigma(Y,Z)}FdYdZ$ as
\begin{align*}
e^{\psi(X)}\Big(\int e^{-2i\sigma(Y,Z)}Fe^{-\psi(X)}\big\{{\tilde\chi}{\hat \chi}+(1-{\tilde\chi}){\hat \chi}\big\}\Big)dYdZ\\
+\int e^{-2i\sigma(Y,Z)}F(1-{\hat\chi})dYdZ.
\end{align*}
Denote ${\hat\chi}_0={\tilde \chi}{\hat \chi}$ as before. 
\begin{lemma}
\label{lem:kanpa}
Let $\Psi(X,Y)=\psi(X+Y)-\psi(X)$ then  
on the support of ${\hat \chi}_0$ one has  
\[
|\Psi(X,Y)|
\leq C\xim^{\tika}g_X^{1/2}(Y).
\]
Assume $a_i\in S^{\lr{s}}_{\delta}(m_i)$, $i=1,2$ then 
\begin{align*}
|\dif_{(x,z)}^{\be}\dif_{(\xi,\zeta)}^{\al}\big(a_1(X+Y)e^{\Psi(X,Y)}a_2(X+Z){\hat\chi}_0\big)|\leq CA^{|\al+\be|}
\xim^{-|\al|+\delta|\be|}\\
\times\big(\xim^{\tika}g_X^{1/2}(Y)+|\al+\be|^{1+\varep}+|\al+\be|^s\xim^{-\delta}\big)^{|\al+\be|}
m_1m_2e^{\Psi}.
\end{align*}
\end{lemma}
Denoting 
\[
{\tilde F}=Fe^{-\psi(X)}=\dif_y^{\nu}\dif_{\eta}^{\mu}(b(X+Y)e^{\Psi(X,Y)})\dif_{\xi}^{\nu}\dif_x^{\mu}a(X+\theta Z)
\]
and  applying Lemma \ref{lem:kanpa} we obatin
\begin{corollary}
\label{cor:kasai}One has 
\begin{align*}
|\dif_y^{\tbe}\dif_{\eta}^{\tal}\dif_{(x,z)}^{\be}\dif_{(\xi,\zeta)}^{\al}{\tilde F}{\hat\chi}_0|
\leq CA^{|\al+\be+\tal+\tbe|}\xim^{\delta|\be+\tbe|-|\al+\tal|}\\
\times \big(\xim^{\tika}+|\tal+\tbe|^{1+\varep}+|\tal+\tbe|^s\xim^{-\delta}\big)^{|\tal+\tbe|}\\
\times \big(\xim^{\tika}g_X^{1/2}(Y)+|\al+\be|^{1+\varep}+|\al+\be|^s\xim^{-\delta}\big)^{|\al+\be|}\\
\times (\xim^{\tika}+p^{1+\varep}+p^s\xim^{-\delta})^p(p^{1+\varep}+p^{s}\xim^{-\delta})^p\xim^{-(1-\delta)p}\xim^{m_1+m_2}e^{\Psi}.
\end{align*}
\end{corollary}
%
Repeating similar arguments proving Lemma \ref{lem:syubu} one has
\begin{lemma}
\label{lem:syubu:b} We have 
\begin{align*}
\Big|\dif_x^{\be}\dif_{\xi}^{\al}\int e^{-2i\sigma(Y,Z)}R_pe^{-\psi(X)}{\hat\chi}dYdZ\Big|\\
\leq  \mu^{\tika p}C_pA^{|\al+\be|}(|\al+\be|^{1+\varep}+|\al+\be|^{s}\xim^{-\delta})^{|\al+\be|}\\
\times \xim^{m_1+m_2-(1-\delta-\tika)p}\xim^{-|\al|+\delta|\be|}.
\end{align*}
\end{lemma}
Write 
\begin{align*}
1-{\hat\chi}=(1-{\hat\chi})\chi(\lr{\eta}^{-1}\lr{\zeta})+(1-{\hat\chi})(1-\chi(\lr{\eta}^{-1}\lr{\zeta}))
={\hat\chi}_2+{\hat\chi}_3
\end{align*}
and repeating similar arguments proving Lemma \ref{lem:syu:b} we have
\begin{lemma}
\label{lem:syu:c}There exist $C>0, A$ and $c>0$ such that
\begin{align*}
\Big|\dif_x^{\be}\dif_{\xi}^{\al}\int e^{-2i\sigma(Y,Z)}R_p(1-{\hat\chi})dYdZ\Big|\\
\leq  CA^{|\al+\be|}(|\al+\be|^{(1+\varep)/(1-\delta)}\big)^{|\al+\be|}e^{-c\xim^{(1-\delta)/(1+\varep)}}.
\end{align*}
\end{lemma}
From \eqref{eq:yunohi} and Lemmas \ref{lem:syubu:b} and \ref{lem:syu:c} we end the proof of Proposition \ref{pro:apsiomega}.

\subsection{Composition $(pe^{-\xim^{\baka}})\# e^{\psi}$}

\begin{proposition}
\label{pro:syu:AAbis}Assume $\psi\in S_{\delta}^{\lr{s}}(\xim^{\tika})$ where $\tika<1-\delta$, $s(1-\delta)<1$ and $p\in S_{0,0}^{(\bas)}(e^{-c\xim^{\baka}})$  with $\bas \tika<1$ and $\baka>\tika$, $c>0$.  Then we have $p\#e^{\psi}, e^{\psi}\#p\in S_{0,0}^{\lr{s^*}}(e^{-c\xim^{\kappa^*}})$ with some $s^*>1$, $\kappa^*>0$ and $c>0$ such that $s^*\tika<1$ and $\kappa^*>\tika$.
\end{proposition}
\begin{corollary}
\label{cor:piko}
Assume $\psi\in S_{\delta}^{\lr{s}}(\xim^{\tika})$ and $p\in S_{0,0}^{(\bas)}(e^{-c\xim^{\baka}})$  with $\bas \tika<1$ and $\baka>\tika$, $c>0$ then $p\#e^{\psi}$, $e^{\psi}\#p\in  S_{\delta}(\xim^l)$ for any $l\in \R$.
\end{corollary}
\begin{remark}
\label{re:sara}\rm As we observed in the proof of Corollary \ref{cor:piko}  if $p\in S_{0,0}^{(s^*)}(e^{-c\xim^{\kappa^*}})$ with some $c>0$  then for any $0<c'<c$ and $\delta\geq 0$ one has  $p\in S_{\delta}^{\lr{\tis}}(e^{-c'\xim^{\kappa^*}})$ with $\tis=s^*+(1+\delta)/\kappa^*$.  
\end{remark}

\noindent
Proof of Proposition \ref{pro:syu:AAbis}:  
Consider  
\begin{align*}
p\# e^{-\psi}=\int e^{-2i\sigma(Y,Z)}p(X+Y)e^{\psi(X+Z)}dYdZ
\end{align*}
where $Y=(y,\eta)$, $Z=(z,\zeta)$. Let ${\hat\chi}$ be as before.  Write
\begin{align*}
\int e^{-2i\sigma(Y,Z)}p(X+Y)e^{\psi(X+Z)}\{{\hat \chi}+(1-{\hat \chi})\}dYdZ
\end{align*}
and consider
\begin{align*}
\dif_x^{\be}\dif_{\xi}^{\al} \int e^{-2i\sigma(Y,Z)}p(X+Y)e^{\psi(X+Z)}{\hat \chi}dYdZ\\
=
\int e^{-2i\sigma(Y,Z)}\lr{\eta}^{-2\ell}\lr{\zeta}^{-2\ell}\lr{D_z}^{2\ell}\lr{D_y}^{2\ell}\\
\times \lr{y}^{-2\ell}\lr{z}^{-2\ell}\lr{D_{\zeta}}^{2\ell}\lr{D_{\eta}}^{2\ell}\dif_x^{\be}\dif_{\xi}^{\al}F{\hat\chi}dYdZ
\end{align*}
where $
F=p(X+Y)e^{\psi(X+Z)}{\hat\chi}$. Since $\psi\in S_{\delta}^{\lr{s}}(\xim^{\tika})$ and $p\in S^{(\bas)}_{0,0}(e^{-c\xim^{\baka}})$ it is clear that
\begin{align*}
|\dif_{\xi}^{\al}\dif_x^{\be}\psi|\leq CA^{|\al+\be|}\xim^{\tika}(|\al+\be|^{1+\varep}+|\al+\be|^s\xim^{-\delta})^{|\al+\be|}\xim^{\delta|\al+\be|},\\
|\dif_{\xi}^{\al}\dif_x^{\be}p|\leq CA^{|\al+\be|}(|\al+\be|^{1+\varep}+|\al+\be|^{\bas}\xim^{-\delta})^{|\al+\be|}\xim^{\delta|\al+\be|}e^{-c\xim^{\baka}}.
\end{align*}
Without restrictions we may assume $\bas\geq s$ and therefore 
\[
|\lr{\eta}^{-2\ell}\lr{\zeta}^{-2\ell}\lr{D_z}^{2\ell}\lr{D_y}^{2\ell}
 \lr{y}^{-2\ell}\lr{z}^{-2\ell}\lr{D_{\zeta}}^{2\ell}\lr{D_{\eta}}^{2\ell}\dif_x^{\be}\dif_{\xi}^{\al}F{\hat\chi}|
\]
is bounded by
\begin{align*}
 C_{\ell}A^{8\ell+|\al+\be|}\lr{\eta}^{-2\ell}\lr{\zeta}^{-2\ell} \lr{y}^{-2\ell}\lr{z}^{-2\ell}
\big(\xim^{\tika}
+(8\ell+|\al+\be|)^{1+\varep}\\
+(8\ell+|\al+\be|)^\bas\xim^{-\delta})^{8\ell+|\al+\be|}
 \xim^{\delta(8\ell+|\al+\be|)}e^{-c'\xim^{\baka}+C\xim^{\tika}}.
\end{align*}
From Lemma \ref{lem:kanta} this can be estimated by
\begin{align*}
C_{\ell}A_1^{|\al+\be|}\lr{\eta}^{-2\ell}\lr{\zeta}^{-2\ell} \lr{y}^{-2\ell}\lr{z}^{-2\ell}
\xim^{8(\tika+\delta)\ell}\\
\times 
 \big(\xim^{\tika+\delta}
+|\al+\be|^{1+\varep}\xim^{\delta}+|\al+\be|^\bas\big)^{|\al+\be|} 
e^{-c''\xim^{\baka}}
\end{align*}
with some $c''>0$. Here we note that
\begin{align*}
\big(\xim^{\tika+\delta}
+|\al+\be|^{1+\varep}\xim^{\delta}+|\al+\be|^{\bas}\big)^{|\al+\be|} 
e^{-c''\xim^{\baka}}\\
\leq CA^{|\al+\be|}|\al+\be|^{{\hat s}|\al+\be|}e^{-c_1\xim^{\baka}}
\end{align*}
where $
{\hat s}=\max\{(\tika+\delta)/\baka, 1+\varep+\delta/\baka, \bas\}$. 
Since $\tika<1-\delta$, $\baka>\tika$ and $\bas \tika<1$ assuming $\varep$ such that $(1+\varep)\tika<1-\delta$ it is clear that ${\hat s}\tika<1$. 
Choosing $\ell>(n+1)/2$ and recalling $\int \Theta^{-\ell}\Phi^{-\ell}dYdZ=C$  we conclude
\begin{equation}
\label{eq:renko}
\Big|\dif_x^{\be}\dif_{\xi}^{\al}\int e^{-2i\sigma(Y,Z)}F{\hat \chi}dYdZ\Big|
\leq \ CA^{|\al+\be|}|\al+\be|^{{\hat s}|\al+\be|}e^{-c''\xim^{\baka}}.
\end{equation}
Denoting $F=p(X+Y)e^{\psi(X+Z)}$ consider  
\begin{equation}
\label{eq:Upac}
\begin{split}
\dif_x^{\be}\dif_{\xi}^{\al}\int e^{-2i\sigma(Y,Z)}F{\hat\chi}_idYdZ=\int e^{-2i\sigma(Y,Z)}\lr{\eta}^{-2N_2}\lr{\zeta}^{-2N_1}\\
\times \lr{D_z}^{2N_2}\lr{D_y}^{2N_1}
 \lr{y}^{-2\ell}\lr{z}^{-2\ell}\lr{D_{\zeta}}^{2\ell}\lr{D_{\eta}}^{2\ell}\dif_x^{\be}\dif_{\xi}^{\al}F{\hat\chi}_idYdZ
\end{split}
\end{equation}
where ${\hat\chi}_1=(1-{\hat\chi})\chi(\lr{\zeta}\lr{\eta}^{-1}/4)$ and  ${\hat\chi}_2=(1-{\hat\chi})(1-\chi(\lr{\zeta}\lr{\eta}^{-1}/4))$. Note that there is $C>0$ such that $\xim\leq C\lr{\eta}$ and $\lr{\zeta}\leq C\lr{\eta}$ on the support of ${\hat \chi}_1$ and hence $\lr{\xi+\zeta}\leq C \lr{\eta}$ there. Similarly on the support of ${\hat\chi}_2$ one has $\xim\leq C\lr{\zeta}$,  $\lr{\eta}\leq C\lr{\zeta}$ and $\lr{\xi+\zeta}\leq C\lr{\zeta}$.

For case ${\hat\chi}_1$ we choose $N_1=\ell$, $N_2=N$. Since $|\psi(X+Z)|\leq C\lr{\eta}^{\tika}$  on the support of ${\hat\chi}_1$ it is not difficult to see that 
\[
\big|\lr{\eta}^{-2N}\lr{\zeta}^{-2\ell}\lr{D_z}^{2N}\lr{D_y}^{2\ell}\lr{y}^{-2\ell}\lr{z}^{-2\ell}\lr{D_{\zeta}}^{2\ell}\lr{D_{\eta}}^{2\ell}\dif_x^{\be}\dif_{\xi}^{\al}F{\hat\chi}_1\big|
\]
is bounded by
\begin{equation}
\label{eq:copi}
\begin{split}
C_{\ell}A^{2N+|\al+\be|}\lr{\eta}^{-2N}\lr{\zeta}^{-2\ell}\lr{y}^{-2\ell}\lr{z}^{-2\ell}\lr{\eta}^{|m|+2\delta\ell+4\tika\ell}\\
\times (C\mu^{\tika}\lr{\eta}^{\tika+\delta}+(2N)^{1+\varep}\lr{\eta}^{\delta}+(2N)^s)^{2N}\\
\times 
 (C\mu^{\tika}\lr{\eta}^{\tika+\delta}+|\al+\be|^{1+\varep}\lr{\eta}^{\delta}+|\al+\be|^{\bas})^{|\al+\be|}
e^{C\lr{\eta}^{\tika}}
\end{split}
\end{equation}
since $\bas\geq s$. Here writing 
\begin{align*}
A^{2N}\lr{\eta}^{-2N}\big( C\mu^{\tika}\lr{\eta}^{\tika+\delta}+(2N)^{1+\varep}\lr{\eta}^{\delta}+(2N)^s\big)^{2N}\\
=\Big(\frac{AC\mu^{\tika} \lr{\eta}^{\tika+\delta}}{\lr{\eta}}+\frac{(2N)^{1+\varep}\lr{\eta}^{\delta}}{\lr{\eta}}+\frac{A(2N)^s}{\lr{\eta}}\Big)^{2N}
\end{align*}
we choose $2N=c_1 \lr{\eta}^{(1-\delta)/(1+\varep)}$ with small $c_1>0$ then the right-hand side is bounded by $Ce^{-c\lr{\eta}^{(1-\delta)/(1+\varep)}}$ because $s(1-\delta)<1$ and $\tika<1-\delta$. Thanks to \eqref{eq:pean} and \eqref{eq:pean:b} it follows that \eqref{eq:copi} is bounded by
\begin{align*}
C_{\ell}A_1^{|\al+\be|}\lr{\zeta}^{-2\ell}\lr{y}^{-2\ell}\lr{z}^{-2\ell}|\al+\be|^{{\hat s}_1|\al+\be|}
e^{-c_1\lr{\eta}^{(1-\delta)/(1+\varep)}}
\end{align*}
where ${\hat s}_1=\max\{(1+\varep)/(1-\delta), \bas\}$ and hence ${\hat s}_1\tika<1$ assuming $(1+\varep)\tika<1-\delta$ which holds for small $\varep$.

For case ${\hat\chi}_2$ we choose $N_1=N$, $N_2=\ell$ in \eqref{eq:Upac}. Since $p\in S_{0,0}^{(\bas)}(e^{-c\xim^{\baka}})$ it follows that 
\[
\big|\lr{\eta}^{-2\ell}\lr{\zeta}^{-2N}\lr{D_z}^{2\ell}\lr{D_y}^{2N}\lr{y}^{-2\ell}\lr{z}^{-2\ell}\lr{D_{\zeta}}^{2\ell}\lr{D_{\eta}}^{2\ell}\dif_x^{\be}\dif_{\xi}^{\al}F{\hat\chi}_2\big|
\]
 is bounded by 
 \begin{equation}
 \label{eq:fubuki}
 \begin{split}
 C_{\ell}A^{2N+|\al+\be|}\lr{\eta}^{-2\ell}\lr{\zeta}^{-2N}\lr{y}^{-2\ell}\lr{z}^{-2\ell}\lr{\zeta}^{|m|+2\delta\ell+2\tika \ell}\\
\times (2N)^{2\bas N}
 (C\lr{\zeta}^{\tika+\delta}+|\al+\be|^{1+\varep}\lr{\zeta}^{\delta}+|\al+\be|^\bas)^{|\al+\be|}
e^{C\lr{\zeta}^{\tika}}.
\end{split}
 \end{equation}
Choose $2N=c_1\lr{\zeta}^{1/\bas}$ then we have
 \[
 \lr{\zeta}^{-2N}(2N)^{2\bas N}=\big(c_1^{\bas}\big)^{2N}\leq e^{-c\lr{\zeta}^{1/\bas}}
 \]
taking $c_1>0$ small.  Here recall that $1/\bas>\tika$ by assumption. Noting 
\begin{align*}
\lr{\zeta}^{(\tika+\delta)|\al+\be|}e^{-c\lr{\zeta}^{1/\bas}}
\leq \big(|\al+\be|^{\bas (\tika+\delta)}\big)^{|\al+\be|}e^{-c'\lr{\zeta}^{1/\bas}}
\end{align*}
it is easy to see that \eqref{eq:fubuki} is bounded by
\begin{align*}
C_{\ell}A_1^{|\al+\be|}\lr{\eta}^{-2\ell}\lr{y}^{-2\ell}\lr{z}^{-2\ell}|\al+\be|^{{\hat s}_2|\al+\be|}e^{-c_1\lr{\zeta}^{1/\bas}}
\end{align*}
where ${\hat s}_2=\max\{\bas(\tika+\delta), \bas\delta+1+\varep,\bas\}$. 
Since $\bas\tika<1$ then $(\bas \delta+1+\varep)\tika<\delta+(1+\varep)\tika< 1$ for small $\varep>0$ and hence ${\hat s}_2\tika<1$. Setting
\[
\left\{\begin{array}{ll}s^*=\max\big\{{\hat s}_1, {\hat s}_2, s\big\},\\
\kappa^*=\min\big\{(1-\delta)/(1+\varep), 1/\bas, \baka\big\}
\end{array}\right.
\]
where $s^*\tika<1$ and $\kappa^*>\tika$ and recalling that $\xim\leq C\lr{\eta}$, $\xim\leq  C\lr{\zeta}$ on the support of ${\hat\chi}_1$ and ${\hat\chi}_2$  one can show 
\begin{equation}
\label{eq:renko:b}
\Big|\dif_x^{\be}\dif_{\xi}^{\al}\int e^{-2i\sigma(Y,Z)}F(1-{\hat \chi})dYdZ\Big|\\
\leq CA^{|\al+\be|}|\al+\be|^{s^*|\al+\be|}
e^{-c_1\xim^{\kappa^*}}.
\end{equation}
Therefore combining \eqref{eq:renko} and \eqref{eq:renko:b} we obtain Proposition \ref{pro:syu:AAbis}.
\subsection{Composition $e^{\lr{D}^{\kappa}}a(x)e^{-\lr{D}^{\kappa}}$}
\begin{proposition}
\label{pro:weyl:1gene} 
Assume $0< \kappa<1$ and
$a(x)\in {\mathcal G}_0^s(\R^n)$ with $s>1$ and $\kappa s< 1$. Then  
the operator $b(x,D)=e^{\lr{D}_{\mu}^{\kappa}}a(x)e^{-\lr{D}_{\mu}^{\kappa}}$ is a pseudodifferential operator with symbol  given by
\begin{equation}
\label{eq:iwama}
b(x,\xi) =a(x)+
\sum_{|\al|=1}
 D_x^{\al}a(x)
\big(\dif_{\xi}\lr{\xi}_{\mu}^{\kappa}   \big)^{\al}
 +q(x,\xi)+r(x,\xi)
\end{equation}
with $q\in S_{\delta}^{\lr{s}}(\xim^{-2+2\kappa})$ and 
$r(x,\xi)\in  S_{0,0}^{(s)}(e^{-c\xim^{1/s}})$ where $\delta=1-\kappa$. 
\end{proposition}
\begin{proof} Write $\phi(\xi)=\lr{\xi}_{{\mu}}^{\kappa}$ then 
\begin{align*}
b(x,\xi)
=\int e^{-2i\sigma(Y,Z)}e^{\phi(\xi+\eta)-\phi(\xi+\zeta)}a(x+y+z)dYdZ.
\end{align*}
The change of variables $
{\tilde y}=y+z$, ${\tilde z}=z-y$, 
${\tilde \eta}=\eta-\zeta$, ${\tilde \zeta}=\zeta+\eta$ 
gives
\begin{equation}
\label{eqn:weyl:7}
\begin{split}
b(x,\xi)=(2\pi )^n
\int e^{-i{\tilde y}{\tilde\eta}}\
e^{\phi(\xi+\frac{{\tilde \eta}}{2})-\phi(\xi-\frac{{\tilde \eta}}{2})}\
a(x+{\tilde y})\
d{\tilde y}d{\tilde \eta}.
\end{split}
\end{equation}
 Insert the Taylor expansion of $a(X+Y)$ 
\begin{align*}
a(x+y)=\sum_{|\al|\leq 1} & \frac{1}{\al !}D_x^{\al}a(x)(iy)^{\al}
+2\sum_{|\al|=2}\frac{(iy)^{\al}}{\al !}\int_0^1(1-\theta)D_x^{\al}a(x+\theta y)d\theta
\end{align*}
into \eqref{eqn:weyl:7} to get
\begin{equation}
\label{eq:okazaki}
\begin{split}
&b(x,\xi)=\sum_{|\al|\leq 1}\frac{1}{\al !}\int e^{-iy\eta}e^{\phi(\xi+\frac{\eta}{2})-\phi(\xi-\frac{\eta}{2})}D_x^{\al}a(x)(iy)^{\al}dyd\eta
\\
+&2\sum_{|\al|=2}\int e^{-iy\eta}e^{\phi(\xi+\frac{\eta}{2})-\phi(\xi-\frac{\eta}{2})}(iy)^{\al}dyd\eta
 \int_0^1(1-\theta)D_x^{\al}a(x+\theta y)d\theta.
\end{split}
\end{equation}
Since $e^{-iy\eta}(iy)^{\al}=(-\partial_{\eta})^{\al}e^{-iy\eta}$ the first term on the right-hand side of \eqref{eq:okazaki} is
\begin{equation}
\label{eqn:weyl:8}
a(x)+\sum_{|\al|=1}\partial_{\eta}^{\al}e^{\phi(\xi+\frac{\eta}{2})-\phi(\xi-\frac{\eta}{2})}\Big|_{\eta=0}D_x^{\al}a(x)
\end{equation}
which is the second term on the right-hand side of \eqref{eq:iwama}. Note that 
\[
\partial_{\xi}^{\al}\phi(\xi)\in S_{\delta}^{\lr{s}}(\xim^{-2+\kappa }),\quad |\al|=1
\]
for any $ \delta\geq 0$. After integration by parts, denoting $H_{\al}(\xi,\eta)=
 \partial_{\eta}^{\al}e^{\phi(\xi+\frac{\eta}{2})
-\phi(\xi-\frac{\eta}{2})}$ the second term on the right-hand side of \eqref{eq:okazaki} yields, up to a multiplicative constant, 
\begin{align*}
R=\int   &  \sum_{|\al|=2}e^{-iy\eta}    H_{\al}(\xi,\eta)dyd\eta\int_0^1(1-\theta)D_x^{\al}a(x+\theta y)d\theta\\
&=\sum_{|\al|=2}\int\int_0^1 e^{ix\eta}(1-\theta)H_{\al}(\xi,\theta\eta)d\eta d\theta\int e^{-iy\eta}D_x^{\al}a(y)dy.
\end{align*}
Denote $E_{\al}(\eta)=\int e^{-iy\eta}D_x^{\al}a(y)dy$ then
\begin{lemma}
\label{le:weyl:5}
There exist $c>0, C>0$ such that $
|E_{\al}(\eta)|
\ \leq\
 C\,
 e^{-c\lr{\eta}^{1/s}}$.
\end{lemma}
\begin{proof} Integration by parts gives
\[
{\eta}^{\be}E_{\al}(\eta)
\ =\
\int e^{-iy\eta}\ D_x^{\al+\be}a(y)\ dy\,.
\]
Then there exist constants $A>0, C>0$ such that  $
|E_{\al}(\eta)|\leq CA^{|\be|}|\be|!^{s}\lr{\eta}^{-|\be|}$. 
Choose $p=|\be|$ such that $p$ minimizes $A^{p}p!^s\lr{\eta}^{-p}$, that is $p\sim e^{-1}A^{-1/s}\lr{\eta}^{1/s}$ so that $A^pp!^s\lr{\eta}^{-p}\lesssim e^{-s^{-1}A^{-1/s}\lr{\eta}^{1/s}}$.
\end{proof}
Note that $H_{\al}(\xi,\eta)$ is a linear combination of terms
\begin{align*}
\partial_{\xi}^{\be_1}\phi(\xi+\frac{\eta}{2})\cdots \partial_{\xi}^{\be_k}\phi(\xi+\frac{\eta}{2})\partial_{\xi}^{{\tilde\be}_1}\phi(\xi-\frac{\eta }{2})&\cdots\partial_{\xi}^{{\tilde\be}_l}\phi(\xi-\frac{\eta}{2})
\
 e^{\phi(\xi+\frac{\eta}{2})-\phi(\xi-\frac{\eta}{2})}
 \\
 =\ &K(\xi,\eta)
 \,e^{\phi(\xi+\frac{\eta}{2})-\phi(\xi-\frac{\eta}{2})}
\end{align*}
where $\sum \be_j=\be$, $\sum {\tilde\be}_j={\tilde\be}$ and $|\be_j|\geq 1$, $|{\tilde\be}_j|\geq 1$, $\be+{\tilde\be}=\al$. Let $\chi(r)\in \gamma^{(1+\varepsilon)}(\R)$ such that $\chi(r)=1$ in $|r|<1$ and $0$ for $r\geq 2$. Then we see 
\begin{equation}
\label{eqn:weyl:10}
|\partial_{\xi}^{\ga}K(\xi,\eta)|
\ \leq\
 CA^{|\ga|}|\ga|!\,\xim^{2\kappa-2}\lr{\xi}_{{\mu}}^{-|\ga|}
\end{equation}
on the support of ${\hat \chi}=\chi(\lr{\xi}^{-1}_{\mu}\lr{\eta})$. Writing $H_{\al}(\xi,\eta)=C_{\al}(\xi,\eta)e^{\phi(\xi+\frac{\eta}{2})-\phi(\xi-\frac{\eta}{2})}$ we have
\[
|\dif_{\xi}^{\ga}\big(C_{\al}(\xi,\eta){\hat\chi}\big)|\leq CA^{\ga|}|\ga|!^{(1+\varepsilon)}\xim^{2\kappa-2}\xim^{-|\ga|}.
\]
On the other hand  it is easy to check that
\[
\big|\dif_{\xi}^{\ga}\big(\phi(\xi+\eta/2)-\phi(\xi-\eta/2)\big)\big|\leq CA^{|\ga|}|\ga|!\xim^{\kappa-1}|\eta|\xim^{-|\ga|}
\]
on the support of ${\hat\chi}$ and in particular 
\begin{align*}
|\phi(\xi+\eta/2)-\phi(\xi-\eta/2)|\leq C\xim^{\kappa-1}|\eta|
\leq C\big(\xim^{-1}\lr{\eta}\big)^{1-\kappa}\lr{\eta}^{\kappa}\leq C'\lr{\eta}^{\kappa}.
\end{align*}
 Therefore from Lemma \ref{lem:seikei} we conclude
\[
\big|\dif_{\xi}^{\ga}\big(H_{\al}(\xi,\eta){\hat\chi}\big)\big|\leq CA^{|\ga|}\xim^{2\kappa-2}\big(\xim^{\kappa-1}|\eta|+|\ga|^{1+\varepsilon}\big)^{|\ga|}\xim^{-|\ga|}e^{C'\lr{\eta}^{\kappa}}
\]
on the support of ${\hat\chi}$. 
From this it follows that
\begin{equation}
\label{eq:aoki}
\begin{split}
\big|\dif_{\xi}^{\ga}\big(H_{\al}(\xi,\eta){\hat\chi}\big)&e^{-c\lr{\eta}^{1/s}}\big|\\
\leq CA^{|\ga|}\xim^{2\kappa-2}&\big(|\ga|^{1+\varepsilon}+|\ga|^s\xim^{-\delta}\big)^{|\ga|}\xim^{-|\ga|}e^{-c'\lr{\eta}^{1/s}}.
\end{split}
\end{equation}
Indeed since $\kappa+\delta=1$ one has 
\begin{align*}
\big(\xim^{\kappa-1}|\eta|\big)^{|\ga|}e^{-c\lr{\eta}^{1/s}}
\leq (\xim^{-\delta}\lr{\eta})^{|\ga|}e^{-c\lr{\eta}^{1/s}}
\leq CA^{|\ga|}(\xim^{-\delta}|\ga|^s)^{|\ga|}e^{-c'\lr{\eta}^{1/s}}.
\end{align*}
With $F_{\al}=(1-\theta)H_{\al}(\xi,\theta\eta)
E_{\al}(\eta)$ denote
 \[
 R=\sum_{|\al|=3}\int\int_0^1 e^{ix\eta}
F_{\al}{\hat\chi}d\eta d\theta
+\sum_{|\al|=3}\int\int_0^1 e^{ix\eta}
F_{\al}(1-{\hat\chi})
d\eta d\theta=R'+R''.
\]
Thanks to Lemma \ref{le:weyl:5} and \eqref{eq:aoki}, taking $|\lr{\eta}^{\be}e^{-c\lr{\eta}^{1/s}}|\leq CA^{|\be|}|\be|^{s|\be|}e^{-c'\lr{\eta}^{1/s}}$ into account ($0<c'<c$) one has
\begin{align*}
\big|\dif_{\xi}^{\ga}\dif_x^{\be}R'\big|\leq CA^{|\be|}\sum_{|\al|=3}\int\int |\dif_{\xi}^{\ga}H(\xi,\theta\eta)|\be|^{s|\be|}e^{-c\lr{\eta}^{1/s}}|d\eta d\theta\\
\leq C\xim^{2\kappa-2}A^{|\ga+\be|}(|\ga|^{1+\varepsilon}+|\ga|^s\xim^{-\delta})^{|\ga|}(\xim^{-\delta}|\be|^s)^{|\be|}\xim^{-|\ga|+\delta |\be|}\\
\leq C\xim^{2\kappa-2}A^{|\ga+\be|}(|\be+\ga|^{1+\varepsilon}+|\be+\ga|^s\xim^{-\delta})^{|\be+\ga|}\xim^{-|\ga|+\delta|\be|}
\end{align*}
for any $\delta\geq 0$ that is $
R'\in S_{\delta}^{\lr{s}}(\xim^{2\kappa-2})$. 
On the other hand it is easy to check that
\begin{align*}
\big|\dif_{\xi}^{\ga}\big(H(\xi,\theta\eta)(1-{\hat\chi})\big)e^{-c\lr{\eta}^{1/s}}\big|
\leq CA^{|\ga|}|\ga|^{(1+\varepsilon)|\ga|}e^{-c\lr{\eta}^{1/s}}
\end{align*}
and hence $|\dif_{\xi}^{\ga}\dif_x^{\be}R''|$ is bounded by $
 CA^{|\be+\ga|}|\be+\ga|^{s|\be+\ga|}e^{-c\xim^{1/s}}$ 
so that $R''\in S_{0,0}^{(s)}(e^{-c\xim^{1/s}})$. Therefore choosing $q(x,\xi)=r+R'$ and $R=R''$ we end the proof of Proposition \ref{pro:weyl:1gene}.
\end{proof}


\end{document}